\documentclass[a4paper,reqno]{amsart}

\usepackage{amssymb,upref}
\usepackage[mathcal]{euscript}
\usepackage[cmtip,all]{xy}

\newcommand{\bydef}{:=}

\newcommand{\veps}{\varepsilon}

\newcommand{\wh}[1]{\widehat{#1}}
\newcommand{\wt}[1]{\widetilde{#1}}
\newcommand{\wb}[1]{\overline{#1}}

\newcommand{\id}{\mathrm{id}}

\newcommand{\gr}{\mathrm{gr}}

\newcommand{\frM}{\mathfrak{M}}

\newcommand{\cA}{\mathcal{A}}

\newcommand{\cD}{\mathcal{D}}

\newcommand{\cF}{\mathcal{F}}

\newcommand{\cI}{\mathcal{I}}

\newcommand{\cK}{\mathcal{K}}
\newcommand{\cL}{\mathcal{L}}
\newcommand{\cM}{\mathcal{M}}

\newcommand{\cO}{\mathcal{O}}

\newcommand{\cR}{\mathcal{R}}

\newcommand{\cV}{\mathcal{V}}
\newcommand{\cW}{\mathcal{W}}

\newcommand{\cZ}{\mathcal{Z}}

\newcommand{\frN}{\mathfrak{N}}

\newcommand{\ZZ}{\mathbb{Z}}

\newcommand{\RR}{\mathbb{R}}

\newcommand{\FF}{\mathbb{F}}
\newcommand{\KK}{\mathbb{K}}
\newcommand{\chr}[1]{\mathrm{char}\,#1}

\DeclareMathOperator{\Hom}{\mathrm{Hom}}
\DeclareMathOperator{\End}{\mathrm{End}}

\DeclareMathOperator{\Aut}{\mathrm{Aut}}
\DeclareMathOperator{\inaut}{\mathrm{Int}}

\DeclareMathOperator{\supp}{\mathrm{Supp}\,}

\newcommand{\Ad}{\mathrm{Ad}}

\newcommand{\Ind}{\mathrm{Ind}}

\newtheorem{theorem}{Theorem}[section]
\newtheorem{proposition}[theorem]{Proposition}

\newtheorem{corollary}[theorem]{Corollary}

\theoremstyle{definition}
\newtheorem{df}[theorem]{Definition}
\newtheorem{example}[theorem]{Example}

\theoremstyle{remark}
\newtheorem{remark}[theorem]{Remark}

\numberwithin{equation}{section}

\newenvironment{romanenumerate}
 {\begin{enumerate}
 
 }{\end{enumerate}}

\begin{document}

\title{Graded simple modules and loop modules}

\author[A.~Elduque]{Alberto Elduque${}^\star$}
\address{Departamento de Matem\'{a}ticas
 e Instituto Universitario de Matem\'aticas y Aplicaciones,
 Universidad de Zaragoza, 50009 Zaragoza, Spain}
\email{elduque@unizar.es}
\thanks{${}^\star$ Supported by the Spanish Ministerio de Econom\'{\i}a y Competitividad---Fondo Europeo de Desarrollo Regional (FEDER) MTM2013-45588-C3-2-P, and by the Diputaci\'on General de Arag\'on---Fondo Social Europeo (Grupo de Investigaci\'on de \'Algebra)}

\author[M. Kochetov]{Mikhail Kochetov${}^\dagger$}
\address{Department of Mathematics and Statistics,
 Memorial University of Newfoundland,
 St. John's, NL, A1C5S7, Canada}
\email{mikhail@mun.ca}
\thanks{${}^\dagger$ Supported by Discovery Grant 341792-2013 of the Natural Sciences and Engineering Research Council (NSERC) of Canada}



\date{}

\begin{abstract}
Necessary and sufficient conditions are given for a $G$-graded simple module over a unital associative algebra, graded by an abelian group $G$, 
to be isomorphic to a loop module of a simple module, as well as for two such loop modules (associated to a subgroup $H$ of $G$) to be isomorphic to each other. 
Under some restrictions, these loop modules are completely reducible (as ungraded modules), and some of their invariants --- inertia group, graded Brauer invariant and Schur index --- which were previously defined for simple modules over graded finite-dimensional semisimple Lie algebras over an algebraically closed field of characteristic zero, are now considered in a more general and natural setting.
\end{abstract}

\maketitle



\section{Introduction}\label{se:intro}

Finite-dimensional graded modules for semisimple Lie algebras over an algebraically closed field $\FF$ of characteristic zero have been studied in \cite{EK_Israel} (see also \cite[Appendix A]{EK_D4}). In particular, the graded simple modules are expressed in terms of simple modules as follows. 

Let $\cL$ be a (finite-dimensional) semisimple Lie algebra over $\FF$ graded by an abelian group $G$, which can be assumed, without loss of generality, to be finitely generated. A $G$-grading $\cV=\bigoplus_{g\in G} \cV_g$ on an $\cL$-module $\cV$ is said to be \emph{compatible} if $\cL_g \cV_h\subseteq \cV_{gh}$ for all $g,h\in G$, that is, if it makes $\cV$ a graded $\cL$-module. Moreover, $\cV$ is said to be \emph{graded simple} if it does not contain any proper nonzero graded submodule. 

The $G$-grading on $\cL$ is given by a homomorphism of algebraic groups $\wh{G}\rightarrow \Aut(\cL)$, $\chi\mapsto \alpha_\chi$, where $\wh{G}=\Hom(G,\FF^\times)$ is the group of characters, so that the homogeneous components are given by $\cL_g=\{ x\in \cL:\alpha_\chi(x)=\chi(g)x\ \forall \chi\in \wh{G}\}$.

Fix a Cartan subalgebra and a system of simple roots for $\cL$, and let $\Lambda^+$ denote the set of dominant integral weights. The group of characters $\wh{G}$ (a quasitorus) acts on the isomorphism classes of irreducible $\cL$-modules, and hence on $\Lambda^+$. For any dominant integral weight $\lambda \in\Lambda^+$, the \emph{inertia group} is 
\[
K_\lambda=\{\chi\in \wh{G}: \chi\ \text{fixes}\ \lambda\}.
\]
Then $K_\lambda$ is Zariski closed in $\wh{G}$ and $[\wh{G}:K_\lambda]$ is finite. Therefore, 
\[
H_\lambda\bydef (K_\lambda)^\perp=\{g\in G: \chi(g)=1\ \forall \chi\in K_\lambda\}
\]
is a finite subgroup of $G$ of size $\lvert H_\lambda\rvert=\lvert \wh{G}\lambda\rvert$, and $K_\lambda$ is canonically isomorphic to the group of characters of $G/H_\lambda$.

Let $\cV_\lambda$ be the simple $\cL$-module with highest weight $\lambda$ and $\varrho_\lambda: U(\cL)\rightarrow \End_\FF(\cV_\lambda)$ the associated representation of the universal enveloping algebra. The $G$-grading on $\cL$ induces naturally a $G$-grading on $U(\cL)$. In general, we cannot expect $\cV_\lambda$ to admit a compatible $G$-grading. However, there is a homomorphism of algebraic groups
\[
\begin{split}
K_\lambda&\longrightarrow \Aut\bigl(\End_\FF(\cV_\lambda)\bigr)\\
\chi&\mapsto \ \tilde\alpha_\chi,
\end{split}
\]
where $\tilde\alpha_\chi\bigl(\varrho_\lambda(x)\bigr)=\varrho_\lambda\bigl(\alpha_\chi(x)\bigr)$ for all $x\in\cL$. This corresponds to a $(G/H_\lambda)$-grading on $\End_\FF(\cV_\lambda)$ such that $\varrho_\lambda$ becomes a homomorphism of $(G/H_\lambda)$-graded algebras, where the $(G/H_\lambda)$-grading on $\cL$ is given by $\cL_{gH_\lambda}=\bigoplus_{h\in H_\lambda}\cL_{gh}$ for all $g\in G$ (a coarsening of the $G$-grading).
The class $[\End_\FF(\cV_\lambda)]$ in the $(G/H_\lambda)$-graded Brauer group (see \cite[\S 2]{EK_Israel} or the end of Section \ref{se:fin_dim_cent} here) is called the \emph{graded Brauer invariant} of $\lambda$, and the degree of the graded division algebra representing $[\End_\FF(\cV_\lambda)]$ is called the \emph{graded Schur index} of $\lambda$.

For each $\wh{G}$-orbit $\cO$ in $\Lambda^+$, select a representative $\lambda$. If $k$ is the graded Schur index of $\lambda$, then the direct sum $\cV_\lambda^k$ of $k$ copies of $\cV_\lambda$ is equipped with a compatible $(G/H_\lambda)$-grading. Finally, let $\cW(\cO)$ be the induced module:
\[
\cW(\cO)= \Ind_{K_\lambda}^{\wh G}\cV_\lambda^k\bydef \FF\wh{G}\otimes_{\FF K_\lambda}\cV_\lambda^k.
\]
(Here $\FF \wh{G}$ denotes the group algebra of $\wh{G}$, and similarly for $\FF K_\lambda$.) Then, by \cite[Theorem 8]{EK_Israel}, up to isomorphism and shift of the grading, these $\cW(\cO)$ are the $G$-graded simple finite-dimensional $\cL$-modules.

In this way, the graded simple modules are obtained very explicitly in terms of simple modules. The inertia groups, graded Brauer invariants and Schur indices have been computed for the classical simple Lie algebras in \cite{EK_Israel} and \cite[Appendix A]{EK_D4}.

A different, though less explicit, approach to describing the graded simple modules is given by Billig and Lau in \cite{BL}, based on the notion of \emph{thin coverings}. It is proved in \cite[Theorem 1.4]{BL} that any graded simple module that contains a maximal (ungraded) submodule can be obtained from a simple module endowed with a thin covering. 

Finally, in a recent work by Mazorchuk and Zhao, it is shown that if $\cL$ is a Lie algebra over an algebraically closed field $\FF$, graded by a finitely generated abelian group $G$, and $\cW$ is a graded simple $\cL$-module with $\dim \cW <\lvert\FF\rvert$, then there is a subgroup $H$ of $G$ and a simple $\cL$-module $\cV$ with a compatible $G/H$-grading satisfying some extra conditions such that $\cW$ is the \emph{loop module} of $\cV$, that is, the subspace
\[
\bigoplus_{g\in G}\cV_{gH}\otimes g\subseteq\cV\otimes_\FF \FF G,
\]
with the $G$-grading given by $\cW_g=\cV_{gH}\otimes g$ and $\cL$-action given by $x(v\otimes g)=xv\otimes g'g$ for all $g,g'\in G$, $x\in\cL_{g'}$ and $v\in\cV_g$ (see \cite[Theorem 31]{MZ}).

\medskip

The aim of this paper is twofold. First we will look at \cite{BL} and \cite{MZ} from a different perspective, based on the results by Allison, Berman, Faulkner and Pianzola \cite{ABFP} on the connection between graded simple algebras and loop algebras of simple, graded algebras. This will allow us to extend the results in \cite{MZ} and, we hope, put them in a more natural context. Second, we will relate this general (and less explicit) approach with our results in \cite{EK_Israel} and \cite{EK_D4}.
Since modules for a Lie algebra are just left modules for its universal enveloping algebra, in this paper, we will mostly work in the more general setting of left modules for an associative algebra. The ground field $\FF$ will be arbitrary unless indicated otherwise.

It should also be noted that any (nonassociative) algebra $\cA$ is a left module for its \emph{multiplication algebra}, that is, the unital associative subalgebra of $\End_\FF(\cA)$ generated by the left and right multiplications by elements of $\cA$. (Recall that the centralizer of this module is called the \emph{centroid} of $\cA$, which is commutative if $\cA=\cA^2$.) Hence, our results in this paper will include the case of graded simple algebras, thus generalizing \cite{ABFP}.

\medskip

The paper is organized as follows. In Section \ref{se:graded_simple}, we recall the main definitions concerning graded algebras and modules. In particular, the centralizer of a graded simple module will play a key role throughout the paper.

Section \ref{se:loop_modules} is devoted to loop modules. Given an abelian group $G$, a $G$-graded unital associative algebra $\cR$ and a subgroup $H$ of $G$, $\cR$ is naturally $G/H$-graded. Let $\pi:G\rightarrow G/H$ be the natural homomorphism. For a left $G/H$-graded left $\cR$-module $\cV$, the loop module $L_\pi(\cV)$ is the subspace $\bigoplus_{g\in G}\cV_{gH}\otimes g\subseteq\cV\otimes\FF G$, which is a $G$-graded left $\cR$-module. It turns out that a $G$-graded left $\cR$-module $\cW$ is isomorphic to a loop module $L_\pi(\cV)$ for a $G/H$-graded module $\cV$ if and only if the graded centralizer of $\cW$ contains a graded subfield that is isomorphic, as a $G$-graded algebra, to the group algebra $\FF H$ (Proposition \ref{pr:LpiVFH}). If the ground field $\FF$ is algebraically closed, all  finite-dimensional $G$-graded simple left $\cR$-modules satisfy this condition and hence are isomorphic to loop modules (see Proposition \ref{pr:Falgclosed}). If, in addition, $H$ is finite and the characteristic of $\FF$ does not divide $\lvert H\rvert$ then the loop module $L_\pi(\cV)$ is isomorphic to the induced module $\FF \wh{G}\otimes_{\FF K} \cV$ (Proposition \ref{pr:loop_induced}).

In Section \ref{se:groupoids}, we define two groupoids, $\frM(\pi)$ and $\frN(\pi)$, for $G$, $H$, $\pi$ and $\cR$ as above. The objects of $\frM(\pi)$ are the simple, central, $G/H$-graded left $\cR$-modules with a certain tightness condition on the grading, while the objects of $\frN(\pi)$ are the pairs $(\cW,\cF)$ where $\cW$ is a $G$-graded simple left $\cR$-module and $\cF$ is a maximal graded subfield of its centralizer that is isomorphic to $\FF H$ as a graded algebra. The loop functor $L_\pi:\frM(\pi)\rightarrow \frN(\pi)$ is defined on objects as $\cV\mapsto \bigl(L_\pi(\cV),L_\pi(\FF 1)\bigr)$ and shown to be faithful and essentially surjective (Theorem \ref{th:MpiNpi}). The objects of $\frM(\pi)$ whose image in $\frN(\pi)$ is isomorphic to $(\cW,\cF)$ are determined explicitly as the \emph{central images} of $(\cW,\cF)$. A certain extension of the loop functor turns out to be an equivalence of categories (Theorem \ref{th:equiv_cat}).

Section \ref{se:fin_dim_cent} deals with the case of finite $H$. We will also assume that the ground field $\FF$ is algebraically closed and its characteristic does not divide $\lvert H\rvert$. Then, for any object $(\cW,\cF)$ of $\frN(\pi)$, $\cW$ is completely reducible as an ungraded $\cR$-module, and the simple submodules of $\cW$ are, up to isomorphism, the central images of $(\cW,\cF)$, so they are endowed with the structure of $G/H$-graded modules (Theorem \ref{th:Wcomp.red.}). Moreover, the isotypic components of $\cW$ are actually $G/Z$-graded simple modules, where $Z\le H$ is the support of the center of the centralizer $C(\cW)$ of $\cW$ (Theorem \ref{th:Deiej} and Corollary \ref{co:Wei}). For any $G$-graded simple left $\cR$-module $\cW$ such that the dimension of $C(\cW)$ is finite and not divisible by $\chr{\FF}$, we will define the \emph{inertia group}, \emph{graded Brauer invariant}, and \emph{graded Schur index}, thus extending the scope of the definitions given in \cite{EK_Israel}. 

In Section \ref{se:simple_centralizer}, we consider a finite-dimensional $G$-graded simple left $\cR$-module $\cW$ and assume that $\FF$ is algebraically closed and $\chr{\FF}$ does not divide the dimension of $C(\cW)$. We prove that if $\cV$ is any simple (ungraded) submodule of $\cW$, then $\End_\FF(\cV)$ is endowed with a unique grading by $G/Z$ compatible with the induced $G/Z$-grading on $\cR$, where $Z$ is the support of the center of $C(\cW)$, and the graded Brauer invariant of $\cW$ is precisely the class $[\End_\FF(\cV)]$ in the $G/Z$-graded Brauer group (Theorem \ref{th:Dcentral} and Corollary \ref{co:Dcentral}). Moreover, $\cV$ is endowed with a structure of $G/H$-graded module for any subgroup $H$ with $Z\leq H\leq G$ such that $\bigoplus_{h\in H}C(\cW)_h$ is a maximal commutative graded subalgebra of $C(\cW)$.

Finally, in Section \ref{se:char0}, assuming $\FF$ algebraically closed of characteristic zero, we give a necessary and sufficient condition for a finite-dimensional simple left $\cR$-module to be a submodule of a $G$-graded simple module (Theorem \ref{th:fin.ind.}). This condition is satisfied in the situation considered in \cite{EK_Israel}, which explains why every finite-dimensional simple module appears as a submodule of a graded simple module in that case.


\section{Graded simple modules}\label{se:graded_simple}

Throughout this work, $\FF$ denotes a ground field, which is arbitrary unless stated otherwise. 
All vector spaces, algebras and modules are assumed to be defined over $\FF$. 
We start by reviewing some basic definitions and facts about gradings. For material not included here, the reader is referred to \cite{EK_mon}.

Let $G$ be an abelian group (written multiplicatively, with neutral element $e$) and let $\cV$ be a vector space. 
A \emph{$G$-grading} on $\cV$ is a vector space decomposition
\begin{equation}\label{eq:Gamma}
\Gamma: \cV=\bigoplus_{g\in G}\cV_g.
\end{equation}
The \emph{support} of $\Gamma$ is the set $\supp(\Gamma)=\{g\in G: \cV_g\neq 0\}$. We will sometimes write $\supp(\cV)$ if $\Gamma$ is fixed.
If $0\ne v\in \cV_g$, $v$ is said to be \emph{homogeneous of degree $g$}, and $\cV_g$ is called the \emph{homogeneous component of degree $g$}. 
If the grading $\Gamma$ is fixed, then $\cV$ will be referred to as a \emph{graded vector space}. A subspace $\cW\subseteq \cV$ is said to be a \emph{graded subspace} if $\cW=\bigoplus_{g\in G}(\cW\cap \cV_g)$. A graded homomorphism of $G$-graded spaces is a linear map that preserves degrees. 

A linear map $f:\cV\rightarrow \cW$ of $G$-graded vector spaces is said to be \emph{homogeneous of degree $g$} if $f(\cV_h)\subseteq \cV_{gh}$ for all $h\in G$. Thus the graded homomorphisms are the homogeneous linear maps of degree $e$. Denote the space of all homogeneous linear maps of degree $g$ by $\Hom_g(\cV,\cW)$ and set $\Hom^{\textup{gr}}(\cV,\cW)\bydef\bigoplus_{g\in G}\Hom_g(\cV,\cW)$. If $\dim \cV$ is finite, then $\Hom^{\textup{gr}}(\cV,\cW)=\Hom(\cV,\cW)$, and thus $\Hom(\cV,\cW)$ becomes $G$-graded.

For a group homomorphism $\alpha:G\rightarrow H$ and a $G$-grading $\Gamma$ as in \eqref{eq:Gamma}, the decomposition ${}^\alpha\Gamma:\cV=\bigoplus_{h\in H}\cV_h'$, where $\cV_h'\bydef \bigoplus_{g\in \alpha^{-1}(h)}\cV_g$, is an $H$-grading on $\cV$, called the \emph{grading induced by $\Gamma$ by means of $\alpha$}. In particular if $\Gamma$ is a $G$-grading on $\cV$, $H$ is a subgroup of $G$ and $\pi:G\rightarrow G/H$ is the natural homomorphism, then ${}^\pi\Gamma$ will be called the \emph{$G/H$-grading induced by $\Gamma$}.

\smallskip

Given an algebra $\cA$ (not necessarily associative), a $G$-grading on $\cA$ is a $G$-grading as a vector space: $\cA=\bigoplus_{g\in G}\cA_g$, that satisfies $\cA_g\cA_h\subseteq \cA_{gh}$ for all $g,h\in G$. Then $\cA$ will be referred to as a \emph{$G$-graded algebra}. Note that if $\cA$ is unital then its identity element $1$ lies in $\cA_e$.

Let now $\cR$ be a unital associative $G$-graded algebra. A \emph{$G$-graded left $\cR$-module} is a left $\cR$-module $\cV$ that is also a $G$-graded vector space, $\cV=\bigoplus_{g\in G}\cV_g$, such that $\cR_g\cV_h\subseteq \cV_{gh}$ for all $g,h\in G$. The $G$-grading on $\cV$ is then said to be {\em compatible} with the $G$-grading on $\cR$. A \emph{$G$-graded right $\cR$-module} is defined similarly. On some occasions, the action of $\cR$ on $\cV$ will not be denoted by juxtaposition, but as $r\cdot v$, $r\bullet v$, etc.; in such cases we will refer to the $\cR$-modules as $(\cV,\cdot)$, $(\cV,\bullet)$, etc.
A \emph{graded submodule} is a graded subspace that is also a submodule. A homomorphism of $G$-graded modules (or $G$-graded homomorphism) is a linear map that is both a homomorphism of modules and of $G$-graded vector spaces. The vector space of graded homomorphisms between two $G$-graded left $\cR$-modules $\cV$ and $\cW$ will be denoted by $\Hom_\cR^G(\cV,\cW)$. Note that $\Hom_\cR^G(\cV,\cW)=\Hom_\cR(\cV,\cW)\cap\Hom_e(\cV,\cW)$. 
We will denote by ${}_\cR\mathrm{Mod}^G$ the abelian category whose objects are the $G$-graded left $\cR$-modules and whose morphisms are the $G$-graded homomorphisms.

We will follow the convention of writing endomorphisms of left modules on the right. Let $\cV$ and $\cW$ be $G$-graded left $\cR$-modules. The space $\Hom_\cR^{\textup{gr}}(\cV,\cW)\bydef \Hom^{\textup{gr}}(\cV,\cW)\cap \Hom_\cR(\cV,\cW)$ is a graded subspace in $\Hom^{\textup{gr}}(\cV,\cW)$. When $\cV=\cW$, we obtain a $G$-graded algebra $C^{\textup{gr}}(\cV)\bydef\Hom_\cR^{\textup{gr}}(\cV,\cV)$, called the \emph{graded centralizer} of $\cV$. Then $\cV$ becomes a graded $\bigl(\cR,C^{\textup{gr}}(\cV)\bigr)$-bimodule.

A nonzero $G$-graded left $\cR$-module $\cV$ is said to be \emph{$G$-graded simple} if its only graded submodules are $0$ and $\cV$. If $G$ or $\cR$ are clear from the context we may omit them and refer to graded left modules, graded simple modules, etc.

\begin{proposition}\label{pr:centralizer}
Let $G$ be an abelian group, let $\cR$ be a unital associative $G$-graded algebra, and let $\cV$ be a graded simple left $\cR$-module. Then its centralizer $C(\cV)\bydef \Hom_\cR(\cV,\cV)$ coincides with its graded centralizer $C^{\textup{gr}}(\cV)$.
\end{proposition}
\begin{proof}
Let $0\ne v\in \cV_g$ be a homogeneous element. Since $\cV$ is graded simple, $\cV=\cR v$. For any $0\ne f\in C(\cV)$, we can write $vf=v_{gg_1}+\cdots +v_{gg_n}$ for some natural number $n$ and some homogeneous elements $v_{gg_i}\in \cV_{gg_i}$. Then, for any $h\in H$ and $r_h\in\cR_h$,
\[
(r_hv)f=r_h(vf)=r_hv_{gg_1}+\cdots +r_hv_{gg_n}.
\]
But $\cR_hv=\cV_{hg}$, so $\cV_{hg}f\subseteq \cR_h\cV_{hgg_1}\oplus\cdots \oplus \cR_h\cV_{hgg_n}$ for all $h\in G$. Therefore, $\cV_hf\subseteq \cV_{hg_1}\oplus\cdots\oplus\cV_{hg_n}$ for all $h\in G$, so $f\in \bigoplus_{i=1}^nC^{\textup{gr}}(\cV)_{g_i}\subseteq C^{\textup{gr}}(\cV)$.
\end{proof}

A unital associative graded algebra is a \emph{graded division algebra} if every nonzero homogeneous element is invertible. Commutative graded division algebras are called \emph{graded fields}. Schur's Lemma shows that the centralizer of a simple module is a division algebra. In the same vein, the graded Schur's Lemma shows that the centralizer of any graded simple module is a graded division algebra (see, for instance, \cite[Lemma 2.4]{EK_mon}), and hence the module is free over its centralizer.

A module $\cV$ is called \emph{central} (or \emph{Schurian}) if its centralizer $C(\cV)$ consists of the scalar multiples of the identity map. Similarly, a graded module $\cV$ will be called \emph{graded central} if $C(\cV)_e$ consists of the scalar multiples of the identity map.

\begin{theorem}\label{th:graded_central}
Let $G$ be an abelian group and let $\cR$ be a unital commutative $G$-graded algebra. Suppose that $\cV$ is a graded simple left $\cR$-module and let $\cD=C(\cV)$.
\begin{romanenumerate}
\item (Graded Density) If $v_1,\ldots,v_n\in\cV$ are homogeneous elements that are linearly independent over $\cD$, then for any $w_1,\ldots,w_n\in\cV$ there exists an element $r\in\cR$ such that $rv_i=w_i$ for $i=1,\ldots,n$.
\item $\cV$ is graded central if and only if $\cV\otimes_\FF\KK$ is a graded simple $(\cR\otimes_\FF\KK)$-module for any field extension $\KK/\FF$.
\end{romanenumerate}
\end{theorem}
\begin{proof}
A proof of (i) appears in \cite[Theorem 2.1]{EK_mon}.

Now, let $\cV$ be graded central and consider a nonzero homogeneous element $v_1\otimes \alpha_1+\cdots +v_n\otimes \alpha_n$ in $\cV\otimes_\FF\KK$, where we may assume $v_1,\cdots,v_n$ linearly independent over $\FF$ and $0\ne \alpha_1,\ldots,\alpha_n\in \KK$. Since $\cD_e=\FF 1$ and $\cD$ is a graded division algebra, each nonzero homogeneous component $\cD_g$ is one-dimensional (for any $0\ne d\in \cD_g$, the map $\cD_e\rightarrow \cD_g$, $x\mapsto xd$, is a bijection). It follows that the elements $v_1,\ldots,v_n$ are linearly independent over $\cD$. For any nonzero homogeneous element $v\in\cV$, by graded density there is an element $r\in \cR$ such that $rv_1=v$ and $rv_i=0$ for $i=2,\ldots,n$. Hence $v\otimes\alpha_1$ belongs to the $\cR$-submodule generated by $v_1\otimes \alpha_1+\cdots +v_n\otimes \alpha_n$. But $\cV\otimes_\FF\KK=(\cR\otimes_\FF\KK)(v\otimes\alpha_1)$, so our arbitrarily chosen nonzero homogeneous element generates $\cV\otimes_\FF\KK$ as a $(\cR\otimes_\FF\KK)$-module, proving graded simplicity.

On the other hand, if $\cV$ is not graded central, then $\cD_e\neq \FF1$, and for any $d\in \cD_e\setminus \FF1$, $\KK=\FF(d)$ is a subfield of $\cD_e$. Then $\cV\otimes_\FF\KK=(\cV\otimes_\KK\KK)\otimes_\FF\KK\simeq \cV\otimes_\KK(\KK\otimes_\FF\KK)$. But $\KK\otimes_\FF\KK$ contains a proper ideal $\cI$ (the kernel of the multiplication $z_1\otimes z_2\mapsto z_1z_2$), and $\cV\otimes_\KK\cI$ is a proper graded submodule of $\cV\otimes_\FF\KK$.
\end{proof}

Later on we will make use not only of $G$-gradings but also of $G$-pregradings, or $G$-coverings (see \cite{S} and \cite{BL}).

\begin{df}\label{df:pregrading}
Let $G$ be an abelian group and let $\cR$ be a unital associative $G$-graded algebra. Let $\cV$ be a left $\cR$-module.
\begin{itemize}
\item A family of subspaces $\Sigma=\{\cV_g:g\in G\}$ is called a \emph{$G$-pregrading} on $\cV$ if $\cV=\sum_{g\in G}\cV_g$ and $\cR_g\cV_h\subseteq \cV_{gh}$ for all $g,h\in G$.
\item Given two pregradings $\Sigma^i=\{\cV_g^i: g\in G\}$, $i=1,2$, $\Sigma^1$ is said to be a \emph{refinement} of $\Sigma^2$ (or $\Sigma^2$ a \emph{coarsening} of $\Sigma^1$) if $\cV_g^1\subseteq \cV_g^2$ for all $g\in G$. If at least one of these containments is strict, the refinement is said to be \emph{proper}.
\item A $G$-pregrading $\Sigma$ is called \emph{thin} if it admits no proper refinement.
\end{itemize}
\end{df}

\begin{example}\label{ex:pregrading_associated}
With $G$, $\cR$ and $\cV$ as in Definition \ref{df:pregrading}, let $H$ be a subgroup of $G$ and consider the induced $G/H$-grading on $\cR$. Write $\wb{G}=G/H$ and $\bar g=gH$ for $g\in G$. Assume that
\begin{equation}\label{eq:Vgbar}
\cV=\bigoplus_{\bar g\in\wb{G}}\cV_{\bar g}
\end{equation}
is a $\wb{G}$-grading on $\cV$ making it a $\wb{G}$-graded left $\cR$-module.
Then the family $\Sigma\bydef\{\cV_g':g\in G\}$, where $\cV_g'=\cV_{\bar g}$ for all $g\in G$, is a $G$-pregrading of $\cV$. 
This is called the \emph{$G$-pregrading associated to the $\wb{G}$-grading} in \eqref{eq:Vgbar}.
\end{example}


\section{Loop modules}\label{se:loop_modules}

For the rest of the paper, $G$ will denote an abelian group, $H$ a subgroup of $G$, $\wb{G}$ will denote the quotient $G/H$, $\bar g=gH$, and $\pi:G\rightarrow \wb{G}$ will denote the natural homomorphism: $\pi(g)=\bar g$ for all $g\in G$. Also, $\cR$ will be a unital associative algebra equipped with a fixed $G$-grading $\Gamma:\cR=\bigoplus_{g\in G}\cR_g$. This $G$-grading on $\cR$ induces the $\wb{G}$-grading ${}^\pi\Gamma$. 

\begin{df}\label{df:loop}
Let $\cV=\bigoplus_{\bar g\in\wb{G}}\cV_{\bar g}$ be a $\wb{G}$-graded left $\cR$-module. The direct sum
\[
L_\pi(\cV)\bydef \bigoplus_{g\in G}\cV_{\bar g}\otimes g\ \Bigl(\subseteq \cV\otimes_\FF \FF G\Bigr)
\]
with the left $\cR$-action given by
\[
r_{g'}(v_{\bar g}\otimes g)=(r_{g'}v_{\bar g})\otimes (g'g)
\]
for all $g,g'\in G$, $r_{g'}\in\cR_{g'}$, and $v_{\bar g}\in\cV_{\bar g}$, is a $G$-graded left $\cR$-module, called the \emph{loop module of $\cV$ relative to $\pi$}.
\end{df}

\begin{remark} The module $L_\pi(\cV)$ is denoted by $M(G,H,\cV)$ in \cite{MZ}. Our notation is inspired by the notation for loop algebras in \cite{ABFP}. Therein, given a $\wb{G}$-graded algebra $\cA$, the \emph{loop algebra} is defined as $L_\pi(\cA)=\bigoplus_{g\in G}\cA_{\bar g}\otimes g$, which is a subalgebra of $\cA\otimes_\FF \FF G$. This generalizes a well known construction in the theory of Kac--Moody Lie algebras.
\end{remark}

\begin{remark}\label{adjoint_functor}
There is a natural `forgetful' functor $F_\pi:{}_\cR\mathrm{Mod}^G\rightarrow {}_\cR\mathrm{Mod}^{\wb{G}}$. The image of a $G$-graded left $\cR$-module $\cW$ is $\cW$ itself with the induced $\wb{G}$-grading ($\cW_{\bar g}\bydef\bigoplus_{h\in H}\cW_{gh}$ for all $g\in G$). The loop construction gives us a functor in the reverse direction: $L_\pi:{}_\cR\mathrm{Mod}^{\wb{G}}\rightarrow {}_\cR\mathrm{Mod}^G$. It is easy to see that $L_\pi$ is the right adjoint of $F_\pi$.
Indeed, for any $G$-graded left $\cR$-module $\cW$ and any $\wb{G}$-graded left $\cR$-module $\cV$, the map
\[
\begin{split}
\Hom_\cR^{\wb{G}}\bigl(F_\pi(\cW),\cV\bigr)&\longrightarrow \Hom_\cR^G\bigl(\cW,L_\pi(\cV)\bigr)\\
 f&\mapsto [\tilde f: w_g\mapsto f(w_g)\otimes g],
\end{split}
\]
for any $g\in G$ and $w_g\in\cW_g$, is a bijection whose inverse is the map
\[
\begin{split}
\Hom_\cR^G\bigl(\cW,L_\pi(\cV)\bigr)&\longrightarrow \Hom_\cR^{\wb{G}}\bigl(F_\pi(\cW),\cV\bigr)\\
 \varphi&\mapsto [\bar\varphi: w_g\mapsto (\id\otimes\epsilon)(\varphi(w_g))],
\end{split}
\]
for any $g\in G$ and $w_g\in \cW_g$, where $\epsilon:\FF G\rightarrow \FF$ is the augmentation map: $\epsilon(g)=1$ for all $g\in G$. 
\end{remark}

A clue to understanding some of the main results in this paper is the following observation:

\begin{proposition}\label{pr:LpiVFH}
A $G$-graded left $\cR$-module $\cW$ is graded isomorphic to a loop module $L_\pi(\cV)$ for a $\wb{G}$-graded left $\cR$-module $\cV$ if and only if its graded centralizer $C^{\textup{gr}}(\cW)$ contains a graded subfield isomorphic to the group algebra $\FF H$.
\end{proposition}
\begin{proof}
If $\cV$ is a $\wb{G}$-graded left $\cR$-module, for any $h\in H$, the linear map $\delta_h:L_\pi(\cV)\rightarrow L_\pi(\cV)$, given by $(v\otimes g)\delta_h=v\otimes (gh)$ for all $g\in G$ and $v\in \cV_{\bar g}=\cV_{\overline{gh}}$, lies in $C^{\textup{gr}}\bigl(L_\pi(\cV)\bigr)$. The linear span of $\{\delta_h:h\in H\}$ is a graded subfield of $C^{\textup{gr}}\bigl(L_\pi(\cV)\bigr)$, isomorphic to the group algebra $\FF H$.

Conversely, if $\cW$ is a $G$-graded left $\cR$-module and $\cF$ is a graded subfield of $C^{\textup{gr}}(\cW)$ isomorphic to $\FF H$, then $\cF=\bigoplus_{h\in H}\FF c_h$, with $c_{h_1}c_{h_2}=c_{h_1h_2}$ for all $h_1,h_2\in H$. Let $\rho:\cF\rightarrow \FF$ be the homomorphism defined by $\rho(c_h)=1$ for all $h\in H$ (that is, the augmentation map if we identify $\cF$ and $\FF H$), and let $\cV=\cW/\cW\ker(\rho)$. Then $\cV$ is naturally $\wb{G}$-graded, because so is $\cW$ (with the $\wb{ G}$-grading induced by the $G$-grading) and $\cW\ker(\rho)$ is a $\wb{G}$-graded submodule of $\cW$. The linear map $\phi:\cW\rightarrow L_\pi(\cV)$ given by $\phi(w)=(w+\cW\ker(\rho))\otimes g$, for all $g\in G$ and $w\in\cW_g$, is a $G$-graded homomorphism. Also, for any $g\in G$ and $h\in H$, $\cW_{gh}=\cW_gc_h$, hence
\[
\cW_{\bar g}=\bigoplus_{h\in H}\cW_{gh}=\cW_g\cF=\cW_g(\FF 1\oplus\ker(\rho)),
\]
so the map
\[
\begin{split}
\cW_g&\longrightarrow \Bigl(\cW/\cW\ker(\rho)\Bigr)_{\bar g}\\
w\ &\mapsto \ w+\cW\ker(\rho),
\end{split}
\]
is a linear isomorphism. Hence $\phi$ is an isomorphism.
\end{proof}

If $\FF$ is algebraically closed then, for many $G$-graded simple modules, the centralizer contains maximal graded subfields that are graded isomorphic to group algebras:

\begin{proposition}\label{pr:Falgclosed}
Let $\cW$ be a $G$-graded simple left $\cR$-module.
\begin{romanenumerate}
\item $C(\cW)=C^\gr(\cW)$ contains maximal graded subfields.
\item If $\cW$ is graded central and $\FF$ is algebraically closed, any graded subfield of $C(\cW)$ is isomorphic to the group algebra of its support.
\item If $\FF$ is algebraically closed and $\dim \cW <\lvert\FF\rvert$ (these may be infinite cardinals), then $\cW$ is graded central.
\end{romanenumerate}
\end{proposition}
\begin{proof}
Recall that $C(\cW)=C^\gr(\cW)$ by Proposition \ref{pr:centralizer}, so $C(\cW)$ is a graded division algebra. For (i), note that $\FF 1$ is a graded subfield of $C(\cW)$ and hence Zorn's Lemma guarantees the existence of maximal graded subfields.

If $\cW$ is graded central, $C(\cW)_e=\FF 1$ and then $\dim C(\cW)_g=1$ for any $g$ in the support of $\cW$. Then if $\FF$ is algebraically closed and $\cF$ is a graded subfield, we have $\cF=\bigoplus_{h\in \supp(\cF)}C(\cW)_h$ and $H\bydef\supp(\cF)$ is a subgroup of $G$. But for all $h\in H$, $C(\cW)_h=\FF x_h$ for some $0\ne x_h$. Then $x_{h_1}x_{h_2}=\sigma(h_1,h_2)x_{h_1h_2}$, where $\sigma:H\times H\rightarrow \FF^\times$ is a symmetric $2$-cocycle. Hence $\cF$ is a commutative twisted group algebra of the abelian group $H$. Since $\FF$ is algebraically closed, $\cF$ is graded isomorphic to the group algebra $\FF H$ (see e.g. \cite[Chapter 1, Lemma 2.9]{Passman}; basically, this is a restatement of the fact that $\mathrm{Ext}(H,\FF^\times)=0$ since $\FF^\times$ is a divisible abelian group).
This proves (ii).

Part (iii) is proved in \cite[Theorem 14]{MZ}.
\end{proof}

The next result is transitivity of the loop construction (as expected in view of Remark \ref{adjoint_functor}):

\begin{proposition}\label{pr:transitive} 
Let $K\leq H\leq G$ and let $\pi':G\rightarrow G/K$ and $\pi'':G/K\rightarrow G/H$ be the natural homomorphisms. Then, for any $G/H$-graded module $\cV$, $L_\pi(\cV)$ and $L_{\pi'}\bigl(L_{\pi''}(\cV)\bigr)$ are isomorphic as $G$-graded modules.
\end{proposition}
\begin{proof}
We have $L_{\pi'}\bigl(L_{\pi''}(\cV)\bigr)= \bigoplus_{g\in G}L_{\pi''}(\cV)_{gK}\otimes g=
\bigoplus_{g\in G}\bigl(\cV_{gH}\otimes gK\bigr)\otimes g$, and hence the map
\[
L_\pi(\cV)=\bigoplus_{g\in G}\cV_{gH}\otimes g\rightarrow L_{\pi'}\bigl(L_{\pi''}(\cV)\bigr),\qquad
v_{gH}\otimes g\mapsto (v_{gH}\otimes gK)\otimes g,
\]
is a $G$-graded isomorphism.
\end{proof}

Under certain conditions, the loop module construction is isomorphic to a well-known construction of induced modules, which we used in \cite{EK_Israel}.
Assume for now that the subgroup $H$ is finite and that $\FF$ is algebraically closed and its characteristic does not divide $n=\lvert H\rvert$. Denote by $\wh{G}$ the group of characters of $G$, that is, group homomorphisms $\chi:G\rightarrow \FF^\times$. The subgroup $H^\perp=\{\chi\in\wh{G}: \chi(h)=1\ \forall h\in H\}$ is naturally isomorphic to the group of characters of $\wb{G}$ (recall that $\wb{G}=G/H$).

Let $\cV$ be a $\wb{G}$-graded left $\cR$-module. Then $\cV$ is a module for the group algebra $\FF(H^\perp)$ with
\[
\chi\cdot v_{\bar g}=\chi(g)v_{\bar g}
\]
for $\chi\in H^\perp$, $g\in G$ and $v_{\bar g}\in \cV_{\bar g}$.

Since $\FF$ is algebraically closed (and hence $\FF^\times$ is a divisible abelian group), any character of $H$ extends to a character of $G$, and since $\chr{\FF}$ does not divide $n$, we have $\lvert \wh{H}\rvert=n$. A transversal of $H^\perp$ in $\wh{G}$ (that is, a set of coset representatives of $H^\perp$ in $\wh{G}$) is any subset $\{\chi_1,\ldots,\chi_n\}$ in $\wh{G}$ such that the restrictions to $H$ are distinct, so $\wh{H}=\{\chi_1\vert_H,\ldots,\chi_n\vert_H\}$. Then $\FF \wh{ G}=\chi_1\FF(H^\perp)\oplus\cdots\oplus\chi_n\FF(H^\perp)$.

The induced $\FF\wh{G}$-module
\[
I_\pi(\cV)\bydef\Ind_{H^\perp}^{\wh{G}}(\cV)=\FF\wh{G}\otimes_{\FF(H^\perp)}\cV=\chi_1\otimes\cV\oplus\cdots\oplus\chi_n\otimes\cV
\]
is a left $\cR$-module by means of 
\begin{equation}\label{eq:induced}
r_{g'}(\chi_j\otimes v_{\bar g})=\chi_j(g')^{-1}\chi_j\otimes r_{g'}v_{\bar g}
\end{equation}
for any $j=1,\ldots,n$, $g,g'\in G$, $r_{g'}\in\cR_{g'}$ and $v_{\bar g}\in\cV_{\bar g}$.

\begin{df}
Given an automorphism $\alpha$ of $\cR$ and a left $\cR$-module $\cV$, we may define a new left $\cR$-module $\cV^\alpha=(\cV,\ast)$ which equals $\cV$ as a vector space, but with the new action given by $r\ast v=\alpha(r) v$. This module $\cV^\alpha$ is referred to as $\cV$ \emph{twisted by the automorphism $\alpha$}.
\end{df}

Equation \eqref{eq:induced} tells us that, as a left $\cR$-module, $I_\pi(\cV)$ is the direct sum of its submodules $\chi_j\otimes \cV$, $j=1,\ldots,n$, and each $\chi_j\otimes\cV$ is isomorphic to the module $\cV$ twisted by $\alpha_{\chi_j}^{-1}$, where the automorphism $\alpha_\chi$, for any $\chi\in \wh{G}$, is given by
\begin{equation}\label{eq:alpha_chi}
\alpha_\chi:\cR\rightarrow \cR,\quad r_g\mapsto \chi(g)r_g,
\end{equation}
for all $g\in G$ and $r_g\in \cR_g$.

Since $\cV$ and its twists are $\wb{G}$-graded, $I_\pi(\cV)$ has a natural $\wb{G}$-grading, with homogeneous component of degree $\bar g$ being $\chi_1\otimes\cV_{\bar g}\oplus\cdots\oplus\chi_n\otimes\cV_{\bar g}$. Clearly, any $\chi\in H^\perp$ acts on this component as the scalar $\chi(\bar g)$, and each of the $\chi_j$ restricts to a diagonalizable operator. It follows that the $\wb{G}$-grading on $I_\pi(\cV)$ can be refined to a $G$-grading:
\begin{equation}\label{eq:G_grad_on_ind}
I_\pi(\cV)_g\bydef\{x\in \chi_1\otimes\cV_{\bar g}\oplus\cdots\oplus\chi_n\otimes\cV_{\bar g}: \chi\cdot x=\chi(g)x\ \forall \chi\in\wh{G}\}.
\end{equation}
(The sum of these subspaces is direct because $\wh{G}$ separates points of $H$, that is, for any $e\ne h\in H$, there exists $\chi\in\wh{G}$ such that $\chi(h)\neq 1$.)

\begin{proposition}\label{pr:loop_induced}
Assume that $H$ is finite, $\FF$ algebraically closed and $\chr{\FF}$ does not divide $n=\lvert H\rvert$. Choose a transversal $\{\chi_1,\ldots,\chi_n\}$ of $H^\perp$ in $\wh{G}$. Let $\cV$ be a $\wb{G}$-graded left $\cR$-module and consider the linear maps
\[
\begin{split}
\varphi:L_\pi(\cV)&\longrightarrow I_\pi(\cV)\\
 v_{\bar g}\otimes g&\mapsto \sum_{j=1}^n\chi_j(g)^{-1}\chi_j\otimes v_{\bar g},
\end{split}
\]
for any $g\in G$ and $v_{\bar g}\in\cV_{\bar g}$, and
\[
\begin{split}
\psi:I_\pi(\cV)&\longrightarrow L_\pi(\cV)\\
 \chi_j\otimes v_{\bar g}&\mapsto \frac{1}{n}\sum_{h\in H}\chi_j(gh)v_{\bar g}\otimes gh,
\end{split}
\]
for any $j=1,\ldots,n$, $g\in G$ and $v_{\bar g}\in\cV_{\bar g}$. Then we have:
\begin{romanenumerate}
\item 
$\varphi$ does not depend on the choice of transversal $\{\chi_1,\ldots,\chi_n\}$;

\item $\varphi$ and $\psi$ are a $G$-graded maps with respect to the $G$-grading on $I_\pi(\cV)$ defined by \eqref{eq:G_grad_on_ind};

\item $\varphi$ and $\psi$ are homomorphisms of $\cR$-modules;

\item $\varphi$ and $\psi$ are inverses of each other;

\item $I_\pi(\cV)$ is a $G$-graded $\cR$-module isomorphic to $L_\pi(\cV)$.
\end{romanenumerate}
\end{proposition}
\begin{proof}
Let $\tilde\chi_j=\chi_j\varpi_j$ where $\varpi_j\in H^\perp$ for any $j=1,\ldots,n$. Then, for any $g\in G$ and $v_{\bar g}\in\cV_{\bar g}$, we have
\[
\begin{split}
\sum_{j=1}^n\tilde\chi_j(g)^{-1}\tilde\chi_j\otimes v_{\bar g}& =
\sum_{j=1}^n\chi_j(g)^{-1}\varpi_j(g)^{-1}\chi_j\otimes \varpi_j\cdot v_{\bar g}\\
&=\sum_{j=1}^n\chi_j(g)^{-1}\varpi_j(g)^{-1}\chi_j\otimes \varpi_j(g)v_{\bar g}
=\sum_{j=1}^n\chi_j(g)^{-1}\chi_j\otimes v_{\bar g}.
\end{split}
\]
This proves (i).

For (ii), we clearly have $\varphi\bigl(L_\pi(\cV)_g\bigr)\subseteq \chi_1\otimes\cV_{\bar g}\oplus\cdots\oplus\chi_n\otimes\cV_{\bar g}$. Let us verify that $\chi\cdot\varphi(v_{\bar g}\otimes g)=\chi(g)\varphi(v_{\bar g}\otimes g)$ for any $g\in G$, $\chi\in\wh{G}$ and $v_{\bar g}\in\cV_g$. Indeed,
\[
\begin{split}
\chi\cdot\varphi(v_{\bar g}\otimes g)
&=\sum_{j=1}^n\chi_j(g)^{-1}\chi\chi_j\otimes v_{\bar g}\\
&=\chi(g)\sum_{j=1}^n(\chi\chi_j)(g)^{-1}(\chi\chi_j)\otimes v_{\bar g}
=\chi(g)\varphi(v_{\bar g}\otimes g),
\end{split}
\]
where we have used (i). Therefore, $\varphi$ is a $G$-graded map. The result for $\psi$ will follow from (iv).

Now consider (iii). For $j=1,\ldots,n$, $g_1,g_2\in G$, $r_{g_1}\in \cR_{g_1}$ and $v_{\bar g_2}\in\cV_{\bar g_2}$, we obtain:
\[
\begin{split}
\varphi\bigl(r_{g_1}(v_{\bar g_2}\otimes g_2)\bigr)&=
  \varphi(r_{g_1}v_{\bar g_2}\otimes g_1g_2)\\
  &=\sum_{j=1}^n\chi_j(g_1g_2)^{-1}\chi_j\otimes r_{g_1}v_{\bar g_2}\\
  &=\sum_{j=1}^n\chi_j(g_1)^{-1}\chi_j(g_2)^{-1}\chi_j\otimes r_{g_1}v_{\bar g_2}\\
  &=r_{g_1}\Bigl(\sum_{j=1}^n\chi_j(g_2)^{-1}\chi_j\otimes v_{\bar g_2}\Bigr)\\
  &=r_{g_1}\varphi(v_{\bar g_2}\otimes g_2),
\end{split}
\]
so $\varphi$ is a homomorphism of $\cR$-modules. The result for $\psi$ will follow from (iv).

Finally, for $j=1,\ldots,n$, $g\in G$ and $v_{\bar g}\in\cV_{\bar g}$,
\[
\begin{split}
\varphi\psi(\chi_j\otimes v_{\bar g})
 &=\frac{1}{n}\sum_{h\in H}\sum_{i=1}^n\chi_i(gh)^{-1}\chi_j(gh)\chi_i\otimes v_{\bar g}\\
 &=\sum_{i=1}^n\chi_i(g)^{-1}\chi_j(g)\Bigl(\frac{1}{n}\sum_{h\in H}\chi_i(h)^{-1}\chi_j(h)\Bigr)\chi_i\otimes v_{\bar g}\\
 &=\chi_j\otimes v_{\bar g},
\end{split}
\]
where the last equality follows from the first orthogonality relation for characters of a finite group, since $\wh{H}=\{\chi_1\vert_H,\ldots,\chi_n\lvert_H\}$. Also,
\[
\begin{split}
\psi\varphi(v_{\bar g}\otimes g)
&=\frac{1}{n}\sum_{h\in H}\sum_{j=1}^n\chi_j(g)^{-1}\chi_j(gh)v_{\bar g}\otimes gh\\
&=\sum_{h\in H}\Bigl(\frac{1}{n}\sum_{j=1}^n\chi_j(h)\Bigr)v_{\bar g}\otimes gh\\
&=v_{\bar g}\otimes g,
\end{split}
\]
by the second orthogonality relation for characters of a finite group. This proves (iv) and hence (v).
\end{proof}

We return to the general setting. Our next result rephrases \cite[Lemma 27]{MZ}.

\begin{proposition}\label{pr:loop_simple}
Let $\cV$ be a $\wb{G}$-graded left $\cR$-module. 
\begin{romanenumerate}
\item If $L_\pi(\cV)$ is $G$-graded simple, then $\cV$ is $\wb{G}$-graded simple.
\item If $\cV$ is $\wb{G}$-graded simple, then $L_\pi(\cV)$ is $G$-graded simple if and only if the $G$-pregrading on $\cV$ associated to its $\wb{G}$-grading (see Definition \ref{df:pregrading} and Example \ref{ex:pregrading_associated}) is thin.
\end{romanenumerate}
\end{proposition}
\begin{proof}
If $\cV$ is not $\wb{G}$-graded simple and $\cW=\bigoplus_{\bar g\in\wb{G}}\cW_{\bar g}$ is a proper $\wb{G}$-graded submodule of $\cV$, then $\bigoplus_{g\in G}\cW_{\bar g}\otimes g$ is a proper $G$-graded submodule of $L_\pi(\cV)$, so (i) follows.

Now if $\cV$ is $\wb{G}$-graded simple but the associated $G$-pregrading is not thin, let $\Sigma=\{\cW_g:g\in G\}$ be a proper refinement. Then again $\bigoplus_{g\in G}\cW_g\otimes g$ is a proper $G$-graded submodule of $L_\pi(\cV)$.

Conversely, if $\cV$ is $\wb{G}$-graded simple but $L_\pi(\cV)$ is not $G$-graded simple, then there is a proper $G$-graded submodule $\bigoplus_{g\in G}\cW_g\otimes g$ of $L_\pi(\cV)$, with $\cW_g\subseteq \cV_{\bar g}$ for all $g\in G$, and at least one of these inclusions is strict. The linear map $\omega:L_\pi(\cV)\rightarrow \cV$, $v_{\bar g}\otimes g\mapsto v_{\bar g}$, is a nonzero homomorphism of $\wb{G}$-graded modules. Since $\cV$ is $\wb{G}$-graded simple, $\cV=\omega\bigl(\bigoplus_{g\in G}\cW_g\otimes g\bigr)=\sum_{g\in G}\cW_g$, and hence $\Sigma=\{\cW_g: g\in G\}$ is a proper refinement of the associated $G$-pregrading in $\cV$.
\end{proof}

\begin{example}
Let $\cR=M_2(\FF)$ and let $G$ be the infinite cyclic group generated by an element $g$. Define a $G$-grading on $\cR$ by
\[
\cR_{g^{-1}}=\left\{\left(\begin{smallmatrix} 0&0\\ \alpha&0\end{smallmatrix}\right):\alpha\in \FF\right\},\quad
\cR_{e}=\left\{\left(\begin{smallmatrix} \alpha&0\\ 0&\beta\end{smallmatrix}\right):\alpha,\beta\in \FF\right\},\quad
\cR_{g}=\left\{\left(\begin{smallmatrix} 0&\alpha\\ 0&0\end{smallmatrix}\right):\alpha\in \FF\right\}.
\]
Let $H$ be the subgroup generated by $g^2$, so $\wb{G}=G/H$ is cyclic of order $2$.
Let $\cV=M_{2\times 1}(\FF)=\left\{\left(\begin{smallmatrix} \alpha\\ \beta\end{smallmatrix}\right): \alpha,\beta\in\FF\right\}$. 
Then $\cV$ is $\wb{G}$-graded with
\[
\cV_{\bar e}=\left\{\left(\begin{smallmatrix} 0\\ \alpha\end{smallmatrix}\right):\alpha\in \FF\right\}, \quad
\cV_{\bar g}=\left\{\left(\begin{smallmatrix} \alpha\\ 0\end{smallmatrix}\right):\alpha\in \FF\right\}.
\]
The $\cR$-module $\cV$ is simple (and central), and hence it is $\wb{G}$-graded simple. The corresponding loop module is
\[
L_\pi(\cV)=\bigoplus_{n\in \ZZ}\Bigl(\cV_{\bar e}\otimes g^{2n}\,\oplus \cV_{\bar g}\otimes g^{2n+1}\Bigr),
\]
and each of the subspaces  $\cV_{\bar e}\otimes g^{2n}\,\oplus \cV_{\bar g}\otimes g^{2n+1}$ is a proper $G$-graded submodule. Hence $L_\pi(\cV)$ is not $G$-graded simple.
Note that the $G$-pregrading $\Sigma=\{\cV_g:g\in G\}$, with $\cV_e=\cV_{\bar e}$, $\cV_g=\cV_{\bar g}$, and $\cV_h=0$ for $h\ne e,g$, is a proper refinement of the $G$-pregrading associated to the $\wb{G}$-grading on $\cV$. Hence this pregrading is not thin.
\end{example}

Let $\cV=\bigoplus_{\bar g\in\wb{G}}\cV_{\bar g}$ be a $\wb{G}$-graded left $\cR$-module. Then the graded centralizer $C^{\textup{gr}}(\cV)$ is a $\wb{G}$-graded algebra, so its loop algebra $L_\pi\bigl(C^{\textup{gr}}(\cV)\bigr)$ is a $G$-graded algebra, which acts naturally on the right on $L_\pi(\cV)$ by means of
\[
(v_{\bar g}\otimes g)(d_{\bar g'}\otimes g')\bydef (v_{\bar g}d_{\bar g'})\otimes (gg'),
\]
for any $g,g'\in G$, $v_{\bar g}\in \cV_{\bar g}$ and $d_{\bar g'}\in C(\cV)_{\bar g'}$.
This action centralizes the action of $\cR$.
On the other hand, $L_\pi(\cV)$ is a $G$-graded module, and hence its graded centralizer $C^{\textup{gr}}\bigl(L_\pi(\cV)\bigr)$ is a $G$-graded algebra.
Therefore we can, and will, identify $L_\pi\bigl(C^{\textup{gr}}(\cV)\bigr)$ with a $G$-graded subalgebra of $C^{\textup{gr}}\bigl(L_\pi(\cV)\bigr)$. 

Since $\FF 1$ is a subalgebra of $C(\cV)_{\bar e}$ ($1$ here denotes the identity element of $C(\cV)$, that is, the identity map on $\cV$), $L_\pi(\FF 1)$ is a $G$-graded subalgebra of $L_\pi\bigl(C^{\textup{gr}}(\cV)\bigr)$. Note that $L_\pi(\FF 1)=\bigoplus_{h\in H}\FF(1\otimes h)$ is, up to the natural identification $h\leftrightarrow 1\otimes h$, the group algebra $\FF H$.

\begin{proposition}\label{pr:LpC}
Let $\cV=\bigoplus_{\bar g\in\wb{G}}\cV_{\bar g}$ be a $\wb{G}$-graded left $\cR$-module. Then, considering as above $L_\pi\bigl(C^{\textup{gr}}(\cV)\bigr)$ as a $G$-graded subalgebra of $C^{\textup{gr}}\bigl(L_\pi(\cV)\bigr)$, it coincides with the centralizer in $C^{\textup{gr}}\bigl(L_\pi(\cV)\bigr)$ of $L_\pi(\FF 1)$, that is:
\[
L_\pi\bigl(C^{\textup{gr}}(\cV)\bigr)=\{\delta\in C^{\textup{gr}}\bigl(L_\pi(\cV)\bigr): \delta(1\otimes h)=(1\otimes h)\delta\ \forall h\in H\}.
\]
\end{proposition}
\begin{proof}
For any $d_{\bar g}\in C^{\textup{gr}}(\cV)_{\bar g}$ and $h\in H$, $(d_{\bar g}\otimes g)(1\otimes h)=d_{\bar g}\otimes gh=(1\otimes h)(d_{\bar g}\otimes g)$, so $L_\pi\bigl(C^{\textup{gr}}(\cV)\bigr)$ is contained in the centralizer of $L_\pi(\FF 1)$.

Conversely, if $\delta_g\in C\bigl(L_\pi(\cV)\bigr)_g$, define an operator $\delta_g^k\in \Hom(\cV_{\bar k},\cV_{\overline{kg}})$, for any $k\in G$, by the formula
\[
(v_{\bar k}\otimes k)\delta_g=v_{\bar k}\delta_g^k\otimes kg.
\]
For any $l,k\in G$, $r_l\in \cR_l$ and $v_{\bar k}\in \cV_{\bar k}$ we have
\[
\Bigl(r_l(v_{\bar k}\otimes k)\Bigr)\delta_g=\bigl(r_lv_{\bar k}\otimes lk\bigr)\delta_g=(r_lv_{\bar k})\delta_g^{lk}\otimes lkg,
\]
and also
\[
\Bigl(r_l(v_{\bar k}\otimes k)\Bigr)\delta_g=r_l\Bigl((v_{\bar k}\otimes k)\delta_g\Bigr)=
r_l\Bigl(v_{\bar k}\delta_g^k\otimes kg\Bigr)=r_l(v_{\bar k}\delta_g^k)\otimes lkg.
\]
Thus
\begin{equation}\label{eq:rvkd}
(r_lv_{\bar k})\delta_g^{lk}=r_l(v_{\bar k}\delta_g^k).
\end{equation}
Moreover, if $\delta_g$ commutes with $1\otimes h$ for any $h\in H$ we get
\begin{multline*}
v_{\bar k}\delta_g^{kh}\otimes khg=(v_{\bar k}\otimes kh)\delta_g=\bigl(v_{\bar k}\otimes k)(1\otimes h)\bigr)\delta_g\\
=\bigl((v_{\bar k}\otimes k)\delta_g\bigr)(1\otimes h)=(v_{\bar k}\delta_g^k\otimes kh)(1\otimes h)
=v_{\bar k}\delta_g^k\otimes kgh,
\end{multline*}
and we conclude that $\delta_g^k=\delta_g^{kh}$ for all $g,k\in G$ and $h\in H$. Hence we may define $\delta_g^{\bar k}\in \Hom(\cV_{\bar k},\cV_{\overline{kg}})$ by $\delta_g^{\bar k}=\delta_g^k$. Now \eqref{eq:rvkd} becomes
\[
(r_kv_{\bar k})\delta_g^{\overline{lk}}=r_l(v_{\bar k}\delta_g^{\bar k}),
\]
so the homogeneous map of degree $\bar g$ given by
\[
\bar\delta_g:\cV\rightarrow \cV,\quad
v_{\bar k}\mapsto v_{\bar k}\delta_g^{\bar k},
\]
for any $\bar k\in \wb{G}$ and $v_{\bar k}\in\cV_{\bar k}$, is in $C(\cV)_{\bar g}$, and 
\[
(v_{\bar k}\otimes k)\delta_g=v_{\bar k}\bar\delta_g\otimes kg=(v_{\bar k}\otimes k)(\bar\delta_g\otimes g).
\]
This shows that $\delta_g$ belongs to the subalgebra $L_\pi\bigl(C^{\textup{gr}}(\cV)\bigr)$, as required.
\end{proof}

\begin{corollary}\label{co:LpF}
Let $\cV=\bigoplus_{\bar g\in\wb{G}}\cV_{\bar g}$ be a $\wb{G}$-graded simple left $\cR$-module. Then the following conditions are equivalent:
\begin{romanenumerate}
\item $\cV$ is central;
\item $L_\pi(\FF 1)$ is self-centralized in $C^{\textup{gr}}\bigl(L_\pi(\cV)\bigr)$;
\item $L_\pi(\FF 1)$ is a maximal graded subfield of $C^{\textup{gr}}\bigl(L_\pi(\cV)\bigr)$.
\end{romanenumerate}
\end{corollary}
\begin{proof} (i)$\Rightarrow$(ii) follows from Proposition \ref{pr:LpC}, and (ii)$\Rightarrow$(iii) is clear. Now if (iii) holds, since $C(\cV)=C^{\textup{gr}}(\cV)$ is a graded division algebra (see Proposition \ref{pr:centralizer}), and $L_\pi\bigl(C(\cV)\bigr)$ centralizes $L_\pi(\FF 1)$, any homogeneous element $\delta$ in $L_\pi\bigl(C(\cV)\bigr)\setminus L_\pi(\FF 1)$ satisfies that the algebra generated by $L_\pi(\FF 1)$, $\delta$ and $\delta^{-1}$ is a graded subfield of $C^{\textup{gr}}\bigl(L_\pi(\cV)\bigr)$. By maximality of $L_\pi(\FF 1)$, this shows that $C(\cV)=\FF 1$.
\end{proof}


\section{The groupoids $\frM(\pi)$ and $\frN(\pi)$}\label{se:groupoids}

In the previous section, we saw that many $G$-graded simple left $\cR$-modules can be obtained using the loop module construction. 
To be precise, if $\cW$ is a $G$-graded simple module, $\FF$ is algebraically closed and $\dim\cW<\lvert\FF\rvert$, 
then there exist maximal graded subfields of $C(\cW)$ and they are isomorphic to group algebras (Proposition \ref{pr:Falgclosed}). 
Moreover, if $\cF$ is such a graded subfield, $\supp(\cF)=H$ and $\pi:G\to G/H$ is the natural homomorphism, then $\cW$ is graded isomorphic to $L_\pi(\cV)$ for some $G/H$-graded module $\cV$ (Proposition \ref{pr:LpiVFH}), and this $\cV$ is graded simple (Proposition \ref{pr:loop_simple}) and central (Corollary \ref{co:LpF}). The purpose of this section is to show that $\cV$ is simple as an ungraded module and explore the connection between $\cV$ and $\cW$ --- in particular, investigate to what extent $\cV$ is determined by $\cW$ and $\cF$. We will work over an arbitrary field $\FF$ and it will be convenient to use the language of groupoids. 

\begin{df}
Fix a subgroup $H\le G$ and let $\pi:G\to\wb{G}=G/H$ be the natural homomorphism.
\begin{itemize}
\item $\frM(\pi)$ is the groupoid whose objects are the simple, central, and $\wb{G}$-graded left $\cR$-modules $\cV=\bigoplus_{\bar g\in\wb{G}}\cV_{\bar g}$ such that the $G$-pregrading associated to the $\wb{G}$-grading is thin, and whose morphisms are the $\wb{G}$-graded isomorphisms (in other words, the nonzero morphisms in ${}_\cR\mathrm{Mod}^{\wb{G}}$).

\item $\frN(\pi)$ is the groupoid whose objects are the pairs $(\cW,\cF)$, where $\cW$ is a $G$-graded simple left $\cR$-module and $\cF$ is a maximal graded subfield of $C(\cW)$ isomorphic to the group algebra $\FF H$ as a $G$-graded algebra, and the morphisms $(\cW,\cF)\rightarrow (\cW',\cF')$ are the pairs $(\phi,\psi)$, where $\phi:\cW\rightarrow \cW'$ is an isomorphism of $G$-graded modules (in other words, a nonzero morphism in ${}_\cR\mathrm{Mod}^G$), $\psi:\cF\rightarrow \cF'$ is an isomorphism of $G$-graded algebras, and $\phi(wc)=\phi(w)\psi(c)$ for all $w\in\cW$ and $c\in\cF$.
\end{itemize}
\end{df}

\begin{remark}\label{rm:groupoids} 
Technically speaking, the objects of $\frM(\pi)$ are pairs $(\cV,\Sigma)$, where $\cV$ is a simple, central left $\cR$-module and $\Sigma$ is a $\wb{G}$-grading on $\cV$ that makes it a $\wb{G}$-graded left $\cR$-module. We have preferred to simplify the notation.
Also, for a morphism $(\phi,\psi)$ in $\frN(\pi)$, $\psi$ is determined by $\phi$. Indeed, $\psi$ is the restriction of the graded isomorphism $C(\cW)\to C(\cW')$ sending $\delta\mapsto\phi\delta\phi^{-1}$. Thus, we could define the morphisms $(\cW,\cF)\to(\cW',\cF')$ in $\frN(\pi)$ as $G$-graded isomorphisms $\phi:\cW\to\cW'$ such that $\phi\cF\phi^{-1}=\cF'$. 
Note that if $\phi:\cW\rightarrow \cW'$ is an isomorphism of $G$-graded modules, then $\phi$ induces an isomorphism $(\cW,\cF)\rightarrow(\cW',\phi\cF\phi^{-1})$ in $\frN(\pi)$.
\end{remark}

\begin{example}\label{ex:M2RZ2}
Let $\FF=\RR$ and $\cR=M_2(\RR)$. Let $G=\{e,g\}$ be a cyclic group of order $2$ and $H=G$, so $\wb{G}$ is trivial. Define a $G$-grading on $\cR$ by
\[
\cR_e=\RR\begin{pmatrix} 1&0\\ 0&1\end{pmatrix} +\RR\begin{pmatrix} 0&1\\ -1&0\end{pmatrix},\quad
\cR_g=\RR\begin{pmatrix} 1&0\\ 0&-1\end{pmatrix} +\RR\begin{pmatrix} 0&1\\ 1&0\end{pmatrix}.
\]
Then $\cR$ is a graded division algebra. Let $\cW=\cR$ be the regular left $\cR$-module with the above grading. Then $\cW$ is a $G$-graded simple module whose centralizer $C(\cW)$ is $\cR$ itself, acting by right multiplication. Observe that
\[
\cF^1=\RR\begin{pmatrix} 1&0\\ 0&1\end{pmatrix}+\RR\begin{pmatrix} 1&0\\ 0&-1\end{pmatrix}\quad\text{and}\quad
\cF^2=\RR\begin{pmatrix} 1&0\\0&1\end{pmatrix}+\RR\begin{pmatrix} 0&1\\ 1&0\end{pmatrix},
\]
are two different maximal graded subfields $G$-graded isomorphic to $\FF H$. Hence $(\cW,\cF^1)$ and $(\cW,\cF^2)$, are two different objects in $\frN(\pi)$.
\end{example}

\begin{remark}
Given an object $(\cW,\cF)$ in $\frN(\pi)$ and an element $g\in G$, we may consider the $G$-graded left $\cR$-module $\cW^{[g]}$, which equals $\cW$ as a left $\cR$-module, but whose grading is `shifted' by $g$, that is, $\cW^{[g]}_{ag}\bydef \cW_{a}$ for all $a\in G$. If $g\in\supp(C(\cW))$ then any nonzero $d\in C(\cW)_g$ can be regarded as a graded isomorphism $\cW^{[g]}\rightarrow \cW$ and hence yields an isomorphism $(\cW^{[g]},d\cF d^{-1})\to(\cW,\cF)$ in $\frN(\pi)$.
(Recall that composition in $C(\cW)$ is applied from left to right.)
Conversely, if $\cW^{[g]}$ is graded isomorphic to $\cW$ then $g\in\supp(C(\cW))$.
\end{remark}

The following results are analogous to the results in \cite{ABFP} on loop algebras.

\begin{proposition}\label{pr:VLpV}
If $\cV$ is an object of $\frM(\pi)$, then $\bigl(L_\pi(\cV),L_\pi(\FF 1)\bigr)$ is an object of $\frN(\pi)$.
\end{proposition}
\begin{proof}
If $\cV$ is an object of $\frM(\pi)$, then $L_\pi(\cV)$ is $G$-graded simple by Proposition \ref{pr:loop_simple}, and $L_\pi(\FF 1)$ is a maximal graded subfield of $C\bigl(L_\pi(\cV)\bigr)$ by Corollary \ref{co:LpF}. 
\end{proof}

\begin{df}
Let $(\cW,\cF)$ be an object in $\frN(\pi)$ and let $\rho:\cF\rightarrow \FF$ be a homomorphism of unital algebras. A \emph{$\rho$-specialization} of $(\cW,\cF)$ is a surjective homomorphism of $\cR$-modules $\gamma:\cW\rightarrow \cV$, where $\cV$ is a $\wb{G}$-graded left $\cR$-module, satisfying:
\begin{itemize}
\item $\gamma(wc)=\gamma(w)\rho(c)$, for all $w\in\cW$ and $c\in\cF$,
\item $\gamma(\cW_g)\subseteq \cV_{\bar g}$ for all $g\in G$ (in other words, $\gamma$ is a $\wb{G}$-graded homomorphism).
\end{itemize}
In this case, $\cV$ is said to be a \emph{central image} of the pair $(\cW,\cF)$.
\end{df}

The natural example of a $\rho$-specialization is the quotient map $\cW\rightarrow \cW/\cW\ker(\rho)$. Note that $\ker(\rho)$ is $\wb{G}$-graded trivially, so $\cW/\cW\ker(\rho)$ is endowed with a natural $\wb{G}$-grading: $\bigl(\cW/\cW\ker(\rho)\bigr)_{\bar g}$ consists of the classes modulo $\cW\ker(\rho)$ of the elements in $\bigoplus_{h\in H}\cW_{gh}$. The central image $\cW/\cW\ker(\rho)$ will be referred to as a \emph{canonical central image}.

The next result shows that, essentially, $\rho$-specializations are unique, and that the objects in $\frN(\pi)$ are, up to isomorphism, obtained from objects in $\frM(\pi)$ via the loop construction as in Proposition \ref{pr:VLpV}. The proof uses some of the arguments in Proposition \ref{pr:LpiVFH}.

\begin{theorem}\label{th:specializations}
Let $(\cW,\cF)$ be an object in $\frN(\pi)$, $\rho:\cF\rightarrow\FF$ a homomorphism of unital algebras, $\cV$ a $\wb{G}$-graded left $\cR$-module, and $\gamma:\cW\rightarrow \cV$ a $\rho$-specialization of $(\cW,\cF)$. Then we have the following:
\begin{romanenumerate}
\item For any $g\in G$, $\gamma$ restricts to a linear bijection of $\cW_g$ onto $\cV_{\bar g}$.

\item If $\Theta$ is a transversal of $H$ in $G$ and, for each $g\in G$, $X^g$ is an $\FF$-basis of $\cW_g$, then $X=\bigcup_{g\in\Theta} X^g$ is a homogeneous $\cF$-basis of $\cW$, and $\gamma$ maps $X$ bijectively onto an $\FF$-basis of $\cV$.

\item There is a unique $\wb{G}$-graded isomorphism $\kappa: \cW/\cW\ker(\rho)\rightarrow \cV$ such that $\gamma(w)=\kappa(w+\cW\ker(\rho))$ for any $w\in \cW$.

\item $\bigl(L_\pi(\cV),L_\pi(\FF 1)\bigr)$ is an object of $\frN(\pi)$ isomorphic to $(\cW,\cF)$.

\item $\cV$ is an object of $\frM(\pi)$.
\end{romanenumerate}
\end{theorem}
\begin{proof}
Consider the $\wb{G}$-grading on $\cW$: $\cW=\bigoplus_{\bar g\in\wb{G}}\cW_{\bar g}$ where, for any $g\in G$, $\cW_{\bar g}=\bigoplus_{h\in H}\cW_{gh}$.
Both $\cW\ker(\rho)$ and $\ker(\gamma)$ are $\wb{G}$-graded, and $\cW\ker(\rho)$ is contained in $\ker(\gamma)$. Thus $\gamma$ induces a surjective $\wb{G}$-graded homomorphism $\kappa:\cW/\cW\ker(\rho)\rightarrow \cV$ such that $\gamma(w)=\kappa(w+\cW\ker(\rho))$ for all $w\in\cW$.

Since $\cF$ is a graded field, $\cW$ has a homogeneous $\cF$-basis. Moreover, for $\Theta$ as in (ii), the map $\Bigl(\bigoplus_{g\in \Theta}\cW_g\Bigr)\otimes_\FF\cF\rightarrow \cW$, given by right multiplication by elements of $\cF$, is a $G$-graded linear bijection. Indeed, for any $k\in G$, there is a unique $g\in\Theta$ and $h\in H$ such that $k=gh$, and $\dim\cF_h=1$, so $\cW_g\otimes_\FF \cF$ is mapped bijectively onto $\cW_g\cF=\bigoplus_{h\in H}\cW_{gh}=\cW_{\bar g}$. Since $\cF=\FF 1\oplus \ker(\rho)$, it follows that $\cW_{\bar g}=\cW_g\oplus \cW_g\ker(\rho)$ for any $g\in G$. In particular, $\cV_{\bar g}=\gamma(\cW_{\bar g})=\gamma(\cW_g)$, so $\gamma$ restricts to a surjection $\cW_g\rightarrow \cV_{\bar g}$.

Now define $\phi:\cW\rightarrow L_\pi(\cV)$ by $\phi(w)=\gamma(w)\otimes g$ for all $g\in G$ and $w\in\cW_g$. Since $\phi(\cW_g)=\gamma(\cW_g)\otimes g=\cV_{\bar g}\otimes g$, $\phi$ is a surjective $G$-graded homomorphism. But $\cW$ is $G$-graded simple, so $\phi$ is an isomorphism. Hence, the restriction of $\gamma$ gives a linear bijection $\cW_g\rightarrow \cV_{\bar g}$, proving (i) and (ii). From the previous paragraph, we know that $\cW=\Bigl(\bigoplus_{g\in\Theta}\cW_g\Bigr)\oplus \Bigl(\bigoplus_{g\in\Theta}\cW_g\ker(\rho)\Bigr)$, and now we have shown that $\Bigl(\bigoplus_{g\in\Theta}\cW_g\Bigr)\cap\ker(\gamma)=0$. Since $\bigoplus_{g\in\Theta}\cW_g\ker(\rho)\subseteq\cW\ker(\rho)\subseteq\ker(\gamma)$, we conclude that $\ker(\gamma)=\cW\ker(\rho)$ and part (iii) follows. 

We can choose elements $c_h\in\cF_h$ such that $\rho(c_h)=1$ for all $h\in H$. Then $\cF=\bigoplus_{h\in H}\FF c_h$, and $c_{h_1}c_{h_2}=c_{h_1h_2}$ for all $h_1,h_2\in H$. Then, for any $g\in G$, $w\in \cW_g$ and $h\in H$, we have 
\[
\phi(wc_h)=\gamma(wc_h)\otimes gh=\gamma(w)\otimes gh
=(\gamma(w)\otimes g)(1\otimes h)=\phi(w)(1\otimes h),
\]
so defining $\psi(c_h)=1\otimes h$ for all $h\in H$, we obtain that $(\phi,\psi)$ is an isomorphism $(\cW,\cF)\rightarrow \bigl(L_\pi(\cV),L_\pi(\FF 1)\bigr)$, proving (iv).

Finally, for (v), it is clear that $\cV$ is $\wb{G}$-graded. Also, $\cV$ is central because of (iv) and Corollary \ref{co:LpF}. By Proposition \ref{pr:loop_simple} $\cV$ is $\wb{G}$-graded simple, and the $G$-pregrading associated to its $\wb{G}$-grading is thin. It remains to show that $\cV$ is simple as an ungraded module. In other words, we must prove that $\cW\ker(\rho)$ is a maximal submodule of $\cW$. The next arguments are inspired by \cite[proof of Lemma 19]{MZ}.

Assume that $\cW'$ is a submodule of $\cW$ with $\cW\ker(\rho)\subsetneq \cW'\subsetneq \cW$ and pick an element 
\[
x=x_{g_0}+x_{g_1}+\cdots +x_{g_n}\in \cW'\setminus\cW\ker(\rho)
\]
with $x_{g_i}\in \cW_{g_i}$, $i=0,1,\ldots,n$, and  $n$ minimal. Since $\cW$ is graded simple, $n\geq 1$. We will follow several steps to get a contradiction.

\smallskip

\noindent$\bullet$\quad The cosets $g_iH$, $i=0,1,\ldots, n$, are distinct.

Indeed, if $g_iH=g_jH$ for $i < j$, then $g_j=g_ih$ for some $h\in H$ and then $x_{g_j}=y_{g_i}c_h$ for some $y_{g_i}\in \cW_{g_i}$. But then $x_{g_j}=y_{g_i}+y_{g_i}(c_h-1)\in\cW_{g_i}+\cW\ker(\rho)$, and hence $x$ is congruent to 
\[
x_{g_0}+\cdots+(x_{g_i}+y_{g_i})+\cdots+x_{g_{j-1}}+x_{g_{j+1}}+\cdots +x_{g_n}
\]
modulo $\cW\ker(\rho)$, a contradiction with the minimality of $n$.

\smallskip

\noindent$\bullet$\quad Let $k_i\bydef g_ig_0^{-1}$ for $i=1,\ldots,n$, so $H\ne k_iH\ne k_jH$ for $i\ne j$. Then, for any $g\in G$, any nonzero element in $\cW_g\oplus\cW_{gk_1}\oplus\cdots\oplus\cW_{gk_n}$ does not belong to $\cW\ker(\rho)$.

Indeed, the set $\{g,gk_1,\ldots,gk_n\}$ can be completed to a transversal $\Theta$ of $H$ in $G$, and $\cW=\Bigl(\bigoplus_{g\in\Theta}\cW_g\Bigr)\oplus\cW\ker(\rho)$.

\smallskip

\noindent$\bullet$\quad  For any $g\in G$ and $w_g\in \cW_g$ there are unique elements $w_{gk_i}\in\cW_{gk_i}$, $i=1,\ldots,n$, such that
\[
w_g+w_{gk_1}+\cdots+w_{gk_n}\in\cW'.
\]

First note that if $w_g+w_{gk_1}+\cdots+w_{gk_n}$ and $w_g+w_{gk_1}'+\cdots+w_{gk_n}'$ both lie in $\cW'$, for some $w_{gk_i},w_{gk_i}'\in \cW_{gk_i}$, $i=1,\ldots,n$, then $\sum_{i=1}^n(w_{gk_i}-w_{gk_i}')\in \cW'$. The minimality of $n$ forces $\sum_{i=1}^n(w_{gk_i}-w_{gk_i}')\in \cW\ker(\rho)$.  Hence $\sum_{i=1}^n(w_{gk_i}-w_{gk_i}')=0$, so $w_{gk_i}=w_{gk_i}'$ for all $i=1,\ldots,n$.

Now, $\cW$ is graded simple, so there is an element $r\in\cR_{gg_0^{-1}}$ such that $w_g=rx_{g_0}$, and hence
\[
w_g+rx_{g_1}+\cdots+rx_{g_n}=r(x_{g_0}+x_{g_1}+\cdots+x_{g_n})\in \cW'.
\]
Thus, we can take $w_{gk_i}=rx_{g_i}$. 

\smallskip

\noindent$\bullet$\quad For each $i=1,\ldots,n$, define a homogeneous linear map $d_i:\cW\rightarrow \cW$ (of degree $k_i$) by $w_gd_i=w_{gk_i}$. Clearly, these are well defined and belong to $C(\cW)$. By construction, for any $g\in G$ and $w_g\in \cW_g$, $w_g(1+d_1+\cdots+d_n)$ is the only element in $\cW_g\oplus\cW_{gk_1}\oplus\cdots\oplus \cW_{gk_n}$ that belongs to $\cW'$ and whose homogeneous component of degree $g$ is $w_g$. Hence, for any $h\in H$, both $(w_gc_h)(1+d_1+\cdots+d_n)$ and $\bigl(w_g(1+d_1+\cdots+d_n)\bigr)c_h=w_g(1+d_1+\cdots+d_n)+\bigl(w_g(1+d_1+\cdots+d_n)\bigr)(c_h-1)$ are in $\cW'$, because $c_h-1$ is in $\ker(\rho)$. By uniqueness, we conclude that $d_ic_h=c_hd_i$ for all $i=1,\ldots,n$ and $h\in H$. Hence $d_1,\ldots,d_n$ are in the centralizer of $\cF$ in $C(\cW)$, which contradicts the fact that $\cF$ is a maximal graded subfield of $C(\cW)$.
\end{proof}

\begin{remark}
The proof above shows that if $(\cW,\cF)$ is an object in $\frN(\pi)$, then $\cW$ contains maximal (ungraded) submodules. This is the condition in \cite[Theorem 1.4]{BL} to express $\cW$ in terms of a simple module with a thin pregrading. Here this expression is more explicit.
\end{remark}

The following example lies at the heart of \cite[Example 34]{MZ}.

\begin{example}\label{ex:M2FZ22}
Take $\cR=M_2(\FF)=\cW$, for any field $\FF$ of characteristic different from $2$, and let $G=\langle g,h\rangle$ be the direct product of two cyclic groups of order $2$. Define a $G$-grading on $M_2(\FF)$ as follows:
\[
\cR_e=\FF\begin{pmatrix} 1&0\\ 0&1\end{pmatrix},\quad
\cR_g=\FF\begin{pmatrix} 1&0\\ 0&-1\end{pmatrix},\quad
\cR_h=\FF\begin{pmatrix} 0&1\\ 1&0\end{pmatrix},\quad
\cR_{gh}=\FF\begin{pmatrix} 0&1\\ -1&0\end{pmatrix}.
\]
Then $\cW$ is $G$-graded simple, $C(\cW)$ is $\cR$ acting by right multiplication, and there are three different maximal graded subfields of $C(\cW)$: $\cF_1=\cR_e\oplus\cR_g$, $\cF_2=\cR_e\oplus\cR_h$ and $\cF_3=\cR_e\oplus\cR_{gh}$. If $H_i$ is the support of $\cF_i$ and $\pi_i:G\rightarrow G/H_i$ is the corresponding natural homomorphism, Theorem \ref{th:specializations} gives $G$-graded isomorphisms of $\cW$ onto loop modules $L_{\pi_i}(\cV_i)$, where each $\cV_i$ is graded by a different quotient of $G$.
\end{example}

Let $\xi:\wb{G}\rightarrow G$ be a fixed section of $\pi$ (that is, $\xi$ is a map such that $\pi\xi=\id_{\wb{G}}$). Then $\Theta=\xi(\wb{G})$ is a transversal of $H$ in $G$, and any such transversal is obtained in this way. 

\begin{df}
Given a $\wb{G}$-graded left $\cR$-module $(\cV,\cdot)$ and a character $\chi$ of $H$, we may define a new $\wb{G}$-graded left $\cR$-module $\cV^\chi=(\cV,\bullet)$, which coincides with $\cV$ as a $\wb{G}$-graded vector space, but where the action of $\cR$ is given by
\[
r_g\bullet v_{\bar k}\bydef \chi\bigl(g\xi(\bar k)\xi(\overline{gk})^{-1}\bigr)r_g\cdot v_{\bar k},
\]
for all $g,k\in G$, $r_g\in \cR_g$ and $v_{\bar k}\in \cV_{\bar k}$. This module $\cV^\chi$ is referred to as $\cV$ \emph{twisted by the character $\chi$}.
\end{df} 

It is straightforward to check that $\cV^\chi$ is indeed a left $\cR$-module. 
If we consider a different section $\tilde\xi:\wb{G}\rightarrow G$, the $\wb{G}$-graded isomorphism class of $\cV^\chi$ does not change. 
Indeed, with the new $\tilde \xi$, the action is given by
\[
r_g\diamond v_{\bar k}\bydef \chi\bigl(g\tilde\xi(\bar k)\tilde\xi(\overline{gk})^{-1}\bigr)r_g\cdot v_{\bar k}.
\]
Consider the $\wb{G}$-graded linear isomorphism 
\[
\begin{split}
\varphi:\cV&\longrightarrow \cV\\ 
v_{\bar g}&\mapsto \chi\bigl(\tilde\xi(\bar g)\xi(\bar g)^{-1}\bigr)v_{\bar g},
\end{split}
\] 
for all $\bar g\in \wb{G}$ and $v_{\bar g}\in \cV_{\bar g}$. Then for $r_g\in \cR_g$ and $v_{\bar k}\in \cV_{\bar k}$, we have:
\[
\begin{split}
\varphi(r_g\diamond v_{\bar k})&=
  \chi\bigl(\tilde\xi(\overline{gk})\xi(\overline{gk})^{-1}\bigr)r_g\diamond v_{\bar k}\\
  &=\chi\bigl(\tilde\xi(\overline{gk})\xi(\overline{gk})^{-1}\bigr)
     \chi\bigl(g\tilde\xi(\bar k)\tilde\xi(\overline{gk})^{-1}\bigr)r_g\cdot v_{\bar k}\\
  &=\chi\bigl(g\tilde\xi(\bar k)\xi(\overline{gk})^{-1}\bigr)r_g\cdot v_{\bar k},
\end{split}
\]
while
\[
\begin{split}
r_g\bullet\varphi(v_{\bar k})
	&=\chi\bigl(\tilde\xi(\bar k)\xi(\bar k)^{-1}\bigr)r_g\bullet v_{\bar k}\\
	&=\chi\bigl(\tilde\xi(\bar k)\xi(\bar k)^{-1}\bigr)
	    \chi\bigl(g\xi(\bar k)\xi(\overline{gk})^{-1}\bigr)r_g\cdot v_{\bar k}\\
	&=\chi\bigl(g\tilde\xi(\bar k)\xi(\overline{gk}^{-1}\bigr)r_g\cdot v_{\bar k},
\end{split}
\]
so $\varphi$ is an $\cR$-module isomorphism $(\cV,\diamond)\rightarrow(\cV,\bullet)$.

\begin{proposition}\label{pr:Vxi}
Let $\xi:\wb{G}\rightarrow G$ be a section of $\pi$ and let $(\cV,\cdot)$ be a $\wb{G}$-graded left $\cR$-module. Let $\chi$ be a character of $H$. If $\chi$ extends to a character of $G$ (as is always the case if $\FF$ is algebraically closed), which we also denote by $\chi$, then  $\cV^\chi$ is $\wb{G}$-graded isomorphic to $\cV^{\alpha_\chi}$, where the automorphism $\alpha_\chi$ of $\cR$ is defined by Equation \eqref{eq:alpha_chi}.
\end{proposition}

\begin{proof} 
Recall that the $\cR$-module structure on $\cV^{\alpha_\chi}$ is defined by $r_g\ast v=\chi(g)r_g\cdot v$, for all $g\in G$, $r_g\in\cR_g$ and $v\in\cV$. 
Consider the $\wb{G}$-graded linear isomorphism $\varphi:\cV\rightarrow \cV$, $v_{\bar g}\mapsto \chi\bigl(\xi(\bar g)\bigr)v_{\bar g}$ for all $v_{\bar g}\in\cV_{\bar g}$. Then, for any $g,k\in G$, $r_g\in \cR_g$ and $v_{\bar k}\in\cV_{\bar k}$, we get
\[
\begin{split}
\varphi(r_g\bullet v_{\bar k})
&=\chi\bigl(\xi(\overline{gk})\bigr)\chi\bigl(g\xi(\bar k)\xi(\overline{gk})^{-1}\bigr)r_g\cdot v_{\bar k}
	=\chi\bigl(g\xi(\bar k)\bigr)r_g\cdot v_{\bar k}\\
&=\bigl(\chi(g)r_g\bigr)\bigl(\chi(\xi(\bar k))v_{\bar k}\bigr)
	=\alpha_\chi(r_g)\cdot\varphi(v_{\bar k})=r_g\ast\varphi(v_{\bar k}),
\end{split}
\]
so $\varphi$ is an $\cR$-module isomorphism $\cV^\chi\rightarrow \cV^{\alpha_\chi}$.
\end{proof}

As in the proof of Theorem \ref{th:specializations}, given an object $(\cW,\cF)$ of $\frN(\pi)$ and a homomorphism of unital algebras $\rho:\cF\rightarrow \FF$, we may choose elements $c_h\in\cF$ for any $h\in H$ such that $\rho(c_h)=1$. Then $\cF=\bigoplus_{h\in H}\FF c_h$. Now, for any character $\chi$ of $H$, the map $\rho_\chi:\cF\rightarrow \FF$, $c_h\mapsto \chi(h)$, is another homomorphism of unital algebras. 
In fact, any homomorphism of unital algebras $\cF\rightarrow \FF$ has this form. 

\begin{proposition}\label{pr:twisted}
Let $(\cW,\cF)$ be an object in $\frN(\pi)$, $\rho:\cF\rightarrow \FF$ a homomorphism of unital algebras, and $\chi$ a character of $H$. Then the central image $\cW/\cW\ker(\rho_\chi)$ is $\wb{G}$-graded isomorphic to the central image $\cW/\cW\ker(\rho)$ twisted by the character $\chi$.
\end{proposition}

\begin{proof}
Fix a section $\xi:\wb{G}\rightarrow G$ of $\pi$ and let $\Theta=\xi(\wb{G})$. Then $\cW=\Bigl(\bigoplus_{g\in \Theta}\cW_g\Bigr)\oplus \cW\ker(\rho)$, so the quotient $\cV\bydef\cW/\cW\ker(\rho)$ can be identified with $\bigoplus_{\bar g\in\wb{G}}\cW_{\xi(\bar g)}$ as a $\wb{G}$-graded space. Consider the action of $\cR$ on $\bigoplus_{g\in\Theta}\cW_{g}$ that arises from this identification. For any $g\in G$, $k\in\Theta$, $r_g\in\cR_g$, and $w_k\in\cW_k$, we have $r_gw_k\in \cW_{gk}$. Now if $k'=\xi(\overline{gk})$ then $gk=k'h$ for a unique $h\in H$, and $r_gw_k\in\cW_{k'h}=\cW_{k'}c_h$. Thus $r_gw_k=w_{k'}c_h$ for a unique $w_{k'}\in\cW_{k'}$, and $r_gw_k=w_{k'}+w_{k'}(c_h-1)\in\cW_{k'}+\cW\ker(\rho)$. Hence, the action of $\cR$ on $\bigoplus_{g\in\Theta}\cW_{g}$ is given by
\begin{equation}\label{eq:rgwk}
r_g\cdot w_k=w_{k'}.
\end{equation}
In the same vein, if $\chi\in \wh{H}$ then $\cV'\bydef\cW/\cW\ker(\rho_\chi)$ can be identified again with $\bigoplus_{g\in\Theta}\cW_g$, but now, for $g$, $k$, $k'$, $h$, $r_g$ and $w_k$ as above, we have
\[
r_gw_k=w_{k'}c_h=\chi(h)w_{k'}+w_{k'}(c_h-\chi(h)1)\in\cW_{k'}+\cW\ker(\rho_\chi),
\]
so the $\cR$-action on $\bigoplus_{g\in \Theta}\cW_g$ is given by 
\[
r_g\bullet w_k=\chi(h)w_{k'}=\chi(h)r_g\cdot w_k=\chi\bigl(g\xi(\bar k)\xi(\overline{gk})^{-1}\bigr)r_g\cdot w_k.
\]
The result follows.
\end{proof}

\begin{corollary}\label{co:ci_as_twists}
\begin{romanenumerate}
\item Let $(\cW,\cF)$ be an object in $\frN(\pi)$, $\rho:\cF\rightarrow \FF$ a homomorphism of unital algebras, and $\cV=\cW/\cW\ker(\rho)$ the corresponding canonical central image. Then the central images of $(\cW,\cF)$ are, up to $\wb{G}$-graded isomorphisms, the modules $\cV^\chi$ for $\chi\in \wh{H}$.

\item If the characters of $H$ extend to $G$ (in particular, if $\FF$ is algebraically closed), then the central images of $(\cW,\cF)$ are the modules $\cV^{\alpha_\chi}$ obtained by twisting the canonical central image $\cV$ by the automorphisms $\alpha_\chi$ of $\cR$, $\chi\in \wh{H}$, as in Equation \eqref{eq:alpha_chi}.
\end{romanenumerate}
\end{corollary}

The next result summarizes most of the work of this section in the language of groupoids $\frM(\pi)$ and $\frN(\pi)$. If $\cV$ is an object of $\frM(\pi)$, then, by Proposition \ref{pr:VLpV}, $\bigl(L_\pi(\cV),L_\pi(\FF 1)\bigr)$ is an object of $\frN(\pi)$. Moreover, if $\varphi:\cV\rightarrow \cV'$ is a morphism in $\frM(\pi)$, then $\bigl(L_\pi(\varphi),\iota\bigr)$ is a morphism in $\frN(\pi)$, where $L_\pi(\varphi)(v_{\bar g}\otimes g)=\varphi(v_{\bar g})\otimes g$, for all $g\in G$ and $v_{\bar g}\in \cV_{\bar g}$, and $\iota$ is the identity map on $L_\pi(\FF 1)$. Thus we obtain the \emph{loop functor} $L_\pi:\frM(\pi)\rightarrow \frN(\pi)$. 

\begin{theorem}\label{th:MpiNpi}
The loop functor $L_\pi:\frM(\pi)\rightarrow \frN(\pi)$ has the following properties:
\begin{romanenumerate}
\item $L_\pi$ is faithful, that is, injective on the set of morphisms $\cV\to\cV'$, for any objects $\cV$ and $\cV'$ in $\frM(\pi)$.

\item $L_\pi$ is essentially surjective, that is, any object $(\cW,\cF)$ in $\frN(\pi)$ is isomorphic to $\bigl(L_\pi(\cV),L_\pi(\FF 1)\bigr)$ for some object $\cV$ in $\frM(\pi)$. The objects $\cV$ in $\frM(\pi)$ with this property are, up to isomorphism in $\frM(\pi)$, the central images of $(\cW,\cF)$.

\item If $\cV$ and $\cV'$ are objects in $\frM(\pi)$ such that their images under $L_\pi$ are isomorphic in $\frN(\pi)$, then there is a character $\chi\in \wh{H}$ such that $\cV'$ is isomorphic to $\cV^\chi$ in $\frM(\pi)$.
\end{romanenumerate}
\end{theorem}

\begin{proof}
The map $\epsilon:L_\pi(\FF 1)\rightarrow \FF$ defined by $\epsilon(1\otimes h)=1$ for all $h\in H$ (the augmentation map of $\FF H\simeq L_\pi(\FF 1)$) is a homomorphism of unital algebras, and, for any object $\cV$ in $\frM(\pi)$, the corresponding canonical central image $L_\pi(\cV)/L_\pi(\cV)\ker(\epsilon)$ is $\wb{G}$-graded isomorphic to $\cV$. This is easy to see directly, or we can invoke Theorem \ref{th:specializations}(iii), since $\omega:L_\pi(\cV)\rightarrow \cV$, $v_{\bar g}\otimes g\mapsto v_{\bar g}$, is an $\epsilon$-specialization. Moreover, if $\varphi:\cV\to\cV'$ is a morphism in $\frM(\pi)$ then the following diagram commutes:
\[
\xymatrix{
L_\pi(\cV)\ar[r]^{L_\pi(\varphi)}\ar[d]_\omega&L_\pi(\cV')\ar[d]^{\omega'}\\
\cV\ar[r]^\varphi&\cV'
}
\]
Since $\omega$ is surjective, part (i) follows.

For (ii), let $(\cW,\cF)$ be an object in $\frN(\pi)$. Pick a homomorphism of unital algebras $\rho:\cF\to\FF$ and let $\cV$ be the corresponding canonical central image of $(\cW,\cF)$. Then, by Theorem \ref{th:specializations}(iv,v), $\cV$ is an object in $\frM(\pi)$ and there exists an isomorphism $(\cW,\cF)\rightarrow \bigl(L_\pi(\cV),L_\pi(\FF 1)\bigr)$ in $\frN(\pi)$. In fact, this argument works for any central image $\cV$ of $(\cW,\cF)$. For the converse, observe that, 
if $(\phi,\psi)$ is an isomorphism $(\cW,\cF)\rightarrow (\cW',\cF')$ in $\frN(\pi)$, $\rho':\cF'\rightarrow \FF$ is a homomorphism of unital algebras, and $\gamma':\cW'\rightarrow \cV'$ is a $\rho'$-specialization of $(\cW',\cF')$, then clearly $\gamma'\phi$ is a $(\rho'\psi)$-specialization of $(\cW,\cF)$.
In particular, if $(\phi,\psi)$ is an isomorphism $(\cW,\cF)\rightarrow\bigl(L_\pi(\cV),L_\pi(\FF 1)\bigr)$ then $\gamma=\omega\phi:\cW\rightarrow\cV$ is a $\rho$-specialization of $(\cW,\cF)$ for $\rho=\epsilon\psi$. Hence $\cV$ is a central image of $(\cW,\cF)$, and the proof of part (ii) is complete.

Finally, for (iii), it remains to invoke Corollary \ref{co:ci_as_twists}(i).
\end{proof}

\begin{remark}\label{rm:classification}
If $\FF$ is algebraically closed, we obtain a classification of $G$-graded central simple modules (which include all $G$-graded simple modules of dimension strictly less than the cardinality of $\FF$) up to isomorphism as follows. By Proposition \ref{pr:Falgclosed}, the centralizer of any such module contains a maximal graded subfield $\cF$ isomorphic to $\FF H$ for some subgroup $H\le G$. We partition all $G$-graded central simple modules according to the graded isomorphism class of their centralizer and, for each class, make a choice of $\cF$ (equivalently, of $H$) and let $\pi:G\to\wb{G}=G/H$ be the natural homomorphism. Then Theorem \ref{th:MpiNpi} implies, taking into account Remark \ref{rm:groupoids} and Corollary \ref{co:ci_as_twists}, that the loop functor $L_\pi$ gives a bijection between the  isomorphism classes of $G$-graded central simple modules with a fixed centralizer and the classes of central simple modules, equipped with a $\wb{G}$-grading whose associated $G$-pregrading is thin, under $\wb{G}$-graded isomorphism and twist through the action of $\wh{G}$. 
\end{remark}

Theorem \ref{th:MpiNpi} can be strengthened as follows. First, we define a groupoid $\wt{\frM}(\pi)$ by extending $\frM(\pi)$: keep the same objects, 
but for the morphisms $\cV\to \cV'$ take all pairs $(\varphi,\chi)$ where $\chi\in\wh{H}$ and $\varphi: V^\chi\to V'$ is a morphism in $\frM(\pi)$. Then, we extend the loop functor to $\wt{L}_\pi:\wt{\frM}(\pi)\rightarrow \frN(\pi)$ as follows: define $\wt{L}_\pi(\cV)=\bigl(L_\pi(\cV),L_\pi(\FF 1)\bigr)$ for objects (the same as before) and send a morphism $(\varphi,\chi)$ as above to the pair $(\phi,\psi)$ where $\phi(v_{\bar g}\otimes gh)=\chi(h)\varphi(v_{\bar g})\otimes gh$ for all $v_{\bar g}\in\cV_{\bar g}$, $g\in\Theta$ and $h\in H$ (this takes place of $L_\pi(\varphi)$ in the previous construction of the loop functor) and $\psi(1\otimes h)=\chi(h)1\otimes h$ (this takes place of the identity map $\iota$ in the previous construction), where $\Theta$ is a fixed transversal of $H$ in $G$. It is straightforward to verify $(\phi,\psi)$ thus defined is indeed a morphism in $\frN(\pi)$, that is, $\phi:L_\pi(\cV)\to L_\pi(\cV')$ is an isomorphism of $G$-graded $\cR$-modules and $\phi(v_{\bar g}\otimes ghh')=\phi(v_{\bar g}\otimes gh)\psi(1\otimes h')$ for all $v_{\bar g}\in\cV_{\bar g}$, $g\in\Theta$ and $h,h'\in H$.

\begin{theorem}\label{th:equiv_cat}
The extended loop functor $\wt{L}_\pi:\wt{\frM}(\pi)\rightarrow \frN(\pi)$ is an equivalence of categories.
\end{theorem}

\begin{proof}
We already know that the functor $\wt{L}_\pi$ is essentially surjective. It remains to show that it is full and faithful, that is, gives a bijection from the set of morphisms $\cV\to\cV'$ in $\wt{\frM}(\pi)$ onto the set of morphisms $\bigl(L_\pi(\cV),L_\pi(\FF 1)\bigr)\to\bigl(L_\pi(\cV'),L_\pi(\FF 1)\bigr)$ in $\frN(\pi)$. 

As in the proof of Theorem \ref{th:MpiNpi}, consider the homomorphism of unital algebras $\epsilon:L_\pi(\FF 1)\to\FF$ and the $\epsilon$-specialization $\omega':L_\pi(\cV')\to\cV'$. Now, for any $\chi\in\wh{H}$, define a homomorphism of unital algebras $\epsilon_\chi:L_\pi(\FF 1)\to\FF$ by $1\otimes h\mapsto\chi(h)$ for all $h\in H$ (the linear extension of $\chi$ to $\FF H\simeq L_\pi(\FF 1)$) and also a linear map $\omega_\chi:L_\pi(\cV)\to\cV^\chi$, $v_{\bar g}\otimes gh\mapsto\chi(h)v_{\bar g}$, for all $v_{\bar g}\in\cV_{\bar g}$, $g\in\Theta$ and $h\in H$, which is easily seen to be an $\epsilon_\chi$-specialization (cf. the proof of Proposition \ref{pr:twisted}). 

For a given morphism $(\varphi,\chi):\cV\to\cV'$ in $\wt{\frM}(\pi)$, the definition of $(\phi,\psi)=\wt{L}_\pi(\varphi,\chi)$ implies the commutativity of the following diagram:
\[
\xymatrix{
L_\pi(\cV)\ar[r]^{\phi}\ar[d]_{\omega_\chi}&L_\pi(\cV')\ar[d]^{\omega'}\\
\cV^\chi\ar[r]^\varphi&\cV'
}
\]
In fact, $\phi$ is the unique $G$-graded isomorphism that makes the above diagram commute, because $\omega_\chi$ is injective on each homogeneous component.
Since $\omega_\chi$ is surjective, $\varphi$ is also uniquely determined by $\phi$, while $\chi$ is uniquely determined by $\psi$. Therefore, the functor $\wt{L}_\pi$ is faithful.

Finally, suppose a morphism $(\phi,\psi):\bigl(L_\pi(\cV),L_\pi(\FF 1)\bigr)\to\bigl(L_\pi(\cV'),L_\pi(\FF 1)\bigr)$ in $\frN(\pi)$ is given. 
Since $L_\pi(\FF 1)\simeq\FF H$ as $G$-graded algebras, $\psi$ must have the form $\psi(1\otimes h)=\chi(h)1\otimes h$ for some $\chi\in\wh{H}$. 
Clearly, $\epsilon_\chi=\epsilon\psi$ and hence $\gamma\bydef\omega'\phi$ is an $\epsilon_\chi$-specialization $L_\pi(\cV)\to\cV'$. Also, as mentioned above, $\omega_\chi$ is an $\epsilon_\chi$-specialization $L_\pi(\cV)\to\cV^\chi$. By Theorem \ref{th:specializations}(iii), $\gamma$ and $\omega_\chi$ induce $\wb{G}$-graded isomorphisms with the canonical central image $L_\pi(\cV)/L_\pi(\cV)\ker(\epsilon_\chi)$, hence there exists a $\wb{G}$-graded isomorphism  $\varphi:\cV^\chi\to\cV'$ such that $\gamma=\varphi\omega_\chi$. By construction, this $\varphi$ makes the above diagram commute, which implies that $(\phi,\psi)=\wt{L}_\pi(\varphi,\chi)$, proving that the functor $\wt{L}_\pi$ is full.
\end{proof}


\section{Graded simple modules with finite-dimensional centralizers}\label{se:fin_dim_cent}

\emph{Assume in this section that $H$ is finite, $\FF$ is algebraically closed and its characteristic does not divide $\lvert H\rvert =n$.}

Under these assumptions, the group of characters of $H$ contains precisely $n$ elements: $\wh{H}=\{\chi_1,\chi_2,\ldots,\chi_n\}$, where we choose $\chi_1$ to be the trivial character.  
Moreover, the group algebra $\FF H$ is semisimple, so $\FF H=\FF \zeta_1\oplus\cdots\oplus \FF \zeta_n$, where $\zeta_1,\ldots,\zeta_n$ are orthogonal primitive idempotents with $\chi_i(\zeta_j)=\delta_{ij}$ for $1\leq i,j\leq n$, where we have extended the characters of $H$ to homomorphisms $\FF H\rightarrow \FF$.

Let $(\cW,\cF)$ be an object of $\frN(\pi)$ and let $\rho:\cF\rightarrow \FF$ be a homomorphism of unital algebras. Then the central image $\cW/\cW\ker(\rho)$ is an object in $\frM(\pi)$ so, in particular, it is simple, that is, $\cW\ker(\rho)$ is a maximal (ungraded) submodule of $\cW$.
The homomorphisms of unital algebras $\cF\rightarrow \FF$ are precisely $\rho=\rho_{\chi_1},\rho_{\chi_2},\ldots,\rho_{\chi_n}$.

\begin{theorem}\label{th:Wcomp.red.} 
Let $(\cW,\cF)$ be an object of $\frN(\pi)$. Then $\cW$ is a completely reducible module and there is a decomposition $\cW=\bigoplus_{i=1}^n\cV^i$ such that, for any $i=1,\ldots,n$, $\cV^i$ is an object of $\frM(\pi)$ isomorphic to $\cW/\cW\ker(\rho_{\chi_i})$.
\end{theorem}

\begin{proof}
Clearly, $\ker \chi_1\cap \ldots\cap \ker\chi_n=0$ and, since $\cW$ is a free $\cF$-module,
this shows that $\cW\ker(\rho_{\chi_1}),\ldots,\cW\ker(\rho_{\chi_n})$ are distinct maximal submodules of $\cW$ with 
\[
\cW\ker(\rho_{\chi_1})\cap \cdots\cap \cW\ker(\rho_{\chi_n})=0.
\]
Now the Chinese Remainder Theorem shows that $\cW$ is isomorphic to the direct sum 
\[
\cW/\cW\ker(\rho_{\chi_1})\oplus\cdots\oplus\cW/\cW\ker(\rho_{\chi_n})
\]
as a $\wb{G}$-graded module.
\end{proof}

Note that since $\chi_1=1$, Proposition \ref{pr:twisted} shows that $\cV^i$ is isomorphic to the twisted module $(\cV^1)^{\chi_i}$. Let us reorder the simple submodules $\cV^i$ in Theorem \ref{th:Wcomp.red.}, and hence the characters $\chi_i$, so that $\{\cV^1,\ldots,\cV^m\}$ is a set of representatives of the isomorphism classes of the simple (ungraded) submodules of $\cW$, with $\chi_1=1$. Therefore, there are natural numbers $n_1,\ldots,n_m$ such that $\cW$ is isomorphic (as an ungraded module) to $n_1\cV^1\oplus\cdots \oplus n_m\cV^m$, where $n_i\cV^i$ denotes the direct sum of $n_i$ copies of $\cV^i$. Being objects of $\frM(\pi)$, the modules $\cV^i$ are central, hence the centralizer $\cD=C(\cW)$ is isomorphic to $M_{n_1}(\FF)\times\cdots\times M_{n_m}(\FF)$. In particular, $\dim C(\cW)=n_1^2+\cdots +n_m^2$ is finite.

Since $\FF$ is assumed to be algebraically closed, the finite-dimensional division algebra $\cD_e$ is just $\FF 1$ (in other words, $\cW$ is graded central), and for any $g\in T=\supp(\cD)$, $\dim \cD_g=1$ (see the proof of Theorem \ref{th:graded_central}). Hence, $\cF$ is determined by its support $H$. We can write $\cF=\bigoplus_{h\in H}\cD_h=\bigoplus_{h\in H}\FF c_h$ where the $c_h\in\cD_h$ are scaled to satisfy $\rho(c_h)=1$ for all $h\in H$.

\begin{remark} Under the hypotheses of this section, the canonical central images $\cW/\cW\ker(\rho_{\chi_i})$, $i=1,\ldots,n$, are precisely the simple quotients of $\cW$. Thus, as ungraded modules, the central images of $(\cW,\cF)$ do not depend on the choice of the maximal graded subfield $\cF$, or, equivalently, of the subgroup $H$. 
\end{remark}

We observe that the center $\cZ=Z(\cD)$ has dimension $m$. Clearly, $\cZ$ is a graded subalgebra of $\cD$, and by maximality of $\cF$, $\cZ$ is contained in $\cF$. Denote $Z=\supp(\cZ)$, so $Z\subseteq H$ and $\cZ=\bigoplus_{z\in Z}\cD_z=\bigoplus_{z\in Z}\FF c_z$ is isomorphic to the group algebra $\FF Z$.

\smallskip

Some of the main properties of the centralizer $\cD=C(\cW)$ will be derived from the next result, which is of independent interest.

\begin{proposition}\label{pr:DTHbeta}
Let $\cD$ be a finite-dimensional $G$-graded division algebra over an algebraically closed field $\FF$. Let $T$ be the support of $\cD$, let $H$ be the support of a maximal graded subfield of $\cD$, and let $Z$ be the support of the center $\cZ=Z(\cD)$.
\begin{romanenumerate}
\item $\cD$ is isomorphic to a twisted group algebra $\FF^\sigma T$, where $\sigma:T\times T\rightarrow \FF^\times$ is a $2$-cocycle such that $\sigma(h_1,h_2)=1$ for all $h_1,h_2\in H$.

\item The map $\beta:T\times T\rightarrow \FF^\times$, defined by $\beta(t_1,t_2)=\sigma(t_1,t_2)\sigma(t_2,t_1)^{-1}$, is an alternating bicharacter of $T$, and $H$ is a maximal isotropic subgroup of $T$ (that is, $H$ is a maximal subgroup such that $\beta(h_1,h_2)=1$ for all $h_1,h_2\in H$). Moreover, the radical of $\beta$ is $Z$.

\item The map 
\[
\begin{split}
\tilde\beta:T&\longrightarrow \wh{H},\\
t&\mapsto \beta(t,.)\vert_H\quad \bigl(\text{that is, }h\mapsto \beta(t,h)\bigr),
\end{split}
\]
is a homomorphism with kernel $H$ and image $\{\chi\in\wh{H}: \chi(z)=1\ \forall z\in Z\}$. In particular, $\tilde\beta$ induces an isomorphism $T/H\rightarrow (H/Z)\,\widehat{\null}\,$.

\item We have $\lvert T\rvert\lvert Z \rvert=\lvert H\rvert^2$ (or, equivalently, $\lvert T/Z\rvert=\lvert H/Z\rvert^2$) and $\chr{\FF}$ does not divide $\lvert T/Z\rvert$.

\item $\cD$ is simple if and only if it is central (that is, $\lvert Z\rvert=1$).
\end{romanenumerate}
\end{proposition}
\begin{proof}
We may choose nonzero elements $c_t\in\cD_t$, for any $t\in T$, satisfying the extra condition $c_{h_1}c_{h_2}=c_{h_1h_2}$ for all $h_1,h_2\in H$ (see the proof of Proposition \ref{pr:Falgclosed}). Then $c_{t_1}c_{t_2}=\sigma(t_1,t_2)c_{t_1t_2}$ for a $2$-cocycle $\sigma$ that is trivial on $H$. This proves (i).

The map $\beta$ is determined by the property 
\[
c_{t_1}c_{t_2}=\beta(t_1,t_2)c_{t_2}c_{t_1}\ \forall t_1,t_2\in T,
\]
and is easily shown to be an alternating bicharacter. Its radical is spanned by the elements $t\in T$ such that $c_t$ is in the center, so it equals $Z$. It also follows that the maximal graded subfields of $\cD$ are the subspaces $\cK=\bigoplus_{k\in K}\cD_k$, with $K$ a maximal isotropic subgroup of $T$. Thus we have (ii). 

For any $t\in T$, $\beta(t,H)=1$ if and only if $t\in H$, since $H$ is a maximal isotropic subgroup of $T$. Hence, $\tilde\beta$ is a homomorphism with kernel $H$. Also, $\beta$ induces a nondegenerate alternating bicharacter $\beta':T/Z\times T/Z\rightarrow \FF^\times$. Since any element whose order equals the characteristic of $\FF$ is in the radical of $\beta'$, this forces the characteristic of $\FF$ to be $0$ or coprime to $\lvert T/Z\rvert$. Now $\beta'$ induces, as above, a  homomorphism $\tilde\beta':T/Z\rightarrow (H/Z)\,\widehat{\null}\;$ with kernel $H/Z$. Note that $(H/Z)\,\widehat{\null}\;$ is naturally isomorphic to the subgroup $\{\chi\in \hat H: \chi(z)=1\ \forall z\in Z\}$. Besides, any character $\chi$ of $H/Z$ extends to a character of $T/Z$, and the nondegeneracy of $\beta'$ gives an element $t\in T$ such that $\beta'(tZ,t'Z)=\chi(t'Z)$ for all $t'\in T$. Then $\beta(t,.)\vert_H=\chi$, and we have proved (iii) and a part of (iv).

Now, we have $\lvert H/Z\rvert=\lvert (H/Z)\,\widehat{\null}\,\rvert$ because the characteristic of $\FF$ does not divide $\lvert H/Z\rvert$. By (iii), we get  $\lvert T/H\rvert=\lvert H/Z\rvert$, and the remaining part of (iv) follows.

Finally, (v) is proved in \cite[p.~35]{EK_mon}.
\end{proof}

\begin{corollary}\label{co:WDfinite}
Let $\cW$ be a $G$-graded simple left $\cR$-module where $\cR$ is a unital associative $G$-graded algebra over an algebraically closed field $\FF$. Then the following conditions are equivalent:
\begin{romanenumerate}
\item $C(\cW)$ contains a maximal graded subfield isomorphic to the group algebra of its support $H$ where $|H|$ is finite and not divisible by $\chr{\FF}$;
\item $\dim C(\cW)$ is finite and not divisible by $\chr{\FF}$.
\end{romanenumerate}
Under these conditions, $\cW$ is completely reducible as an ungraded module.
\end{corollary}

\smallskip

Now we go back to the situation described after Theorem \ref{th:Wcomp.red.}. The center $\cZ$ is isomorphic to the direct product of $m$ copies of $\FF$, so there are orthogonal idempotents $\veps_1,\ldots,\veps_m$ such that $\cZ=\FF\veps_1\oplus\cdots\oplus\FF\veps_m$. For each $i$, the idempotent $\veps_i$ corresponds to the identity matrix in the $i$th factor of $\cD\simeq M_{n_1}(\FF)\times\cdots\times M_{n_m}(\FF)$. Hence, 
\begin{equation}\label{eq:WWeis}
\cW=\cW\veps_1\oplus\cdots\oplus\cW\veps_m,
\end{equation} 
where $\cW\veps_i$ is the submodule of $\cW$ equal to the sum of all submodules isomorphic to $\cV^i$. Note that $\cW\veps_i$ is a $G/Z$-graded submodule of $\cW$.

\begin{proposition}\label{pr:ViVj}
For any $1\leq j,k\leq n$, $\cV^j$ is isomorphic to $\cV^k$ if and only if the restrictions $\chi_j\vert_Z$ and $\chi_k\vert_Z$ coincide.
In particular, $\cV^j$ is isomorphic to $\cV^1$ if and only if $\chi_j(z)=1$ for all $z\in Z$.
\end{proposition}

\begin{proof}
By our choice of the ordering, $\cV^j$ is isomorphic to exactly one of the first $m$ direct summands, say, to $\cV^i$ where $1\le i\le m$. Since $\cV^i$ is contained in $\cW\veps_i$, we have $\cV^i=\cV^i\veps_i$ and $\cV^i\veps_l=0$ for all $1\le l\le m$, $l\ne i$. Hence, $\cV^j$ is isomorphic to $\cV^i$ if and only if $\cV^j$ is contained in $\cW\veps_i$, if and only if $\cV^j\veps_l=0$ for all $1\le l\le m$, $l\ne i$. 

This unique $i$ can also be characterized in terms of $\rho_{\chi_j}$. Indeed, $\cV^j$ is isomorphic to $\cW/\cW\ker(\rho_{\chi_j})$, and there exists $1\le i\le m$ such that $\rho_{\chi_j}(\veps_i)=1$ and $\rho_{\chi_j}(\veps_l)=0$ for $l\ne i$. It follows that $\cW\ker(\rho_{\chi_j})$ contains $\bigoplus_{l\ne i}\cW\veps_l$, so  $\cW/\cW\ker(\rho_{\chi_j})$ is isomorphic to a simple quotient of $\cW/\bigl(\bigoplus_{l\ne i}\cW\veps_l\bigr)\simeq \cW\veps_i$, hence isomorphic to $\cV^i$.
Therefore, $\cV^j$ is isomorphic to $\cV^i$ if and only if $\rho_{\chi_j}(\veps_i)=1$ and $\rho_{\chi_j}(\veps_l)=0$ for all $1\le l\le m$, $l\ne i$, if and only if $\rho_{\chi_j}\vert_{\cZ}=\rho_{\chi_i}\vert_{\cZ}$. But $\rho_{\chi_j}(c_z)=\chi_j(z)$ and $\rho_{\chi_i}(c_z)=\chi_i(z)$, so $\rho_{\chi_j}\vert_{\cZ}=\rho_{\chi_i}\vert_{\cZ}$ if and only if $\chi_j\vert_Z=\chi_i\vert_Z$. The result follows.
\end{proof}

\begin{corollary}\label{co:V1alpha}
For any character $\chi\in\wh{G}$, the twisted module $(\cV^1)^{\alpha_\chi}$, where the automorphism $\alpha_\chi$ is given by Equation \eqref{eq:alpha_chi}, is isomorphic to $\cV^i$ for the unique $1\le i\le m$ such that $\chi\vert_Z=\chi_i\vert_Z$.
In particular, $(\cV^1)^{\alpha_\chi}$ is isomorphic to $\cV^1$ if and only if $\chi\in Z^\perp$ (that is, $\chi(z)=1$ for all $z\in Z$).
\end{corollary}

\begin{proof}
With $\chi'=\chi\vert_H$, $(\cV^1)^{\alpha_\chi}$ is the twisted module $(\cV^1)^{\chi'}$ (Proposition \ref{pr:Vxi}), which is isomorphic to $\cW/\cW\ker(\rho_{\chi'})$ (Proposition \ref{pr:twisted}). Thus, $(\cV^1)^{\alpha_\chi}$ is isomorphic to $\cV^i$ if and only if $\chi'\vert_Z=\chi_i\vert_Z$.
\end{proof}

The graded division algebra $\cD=C(\cW)$ is, on the one hand, the direct sum $\cD=\bigoplus_{t\in T}\cD_t$, where $\dim\cD_t=1$ for all $t\in T$, and, on the other hand, $\cD=\bigoplus_{i=1}^m\cD\veps_i$, where $\cD\veps_i$ is a simple ideal isomorphic to the matrix algebra $M_{n_i}(\FF)$ for all $i=1,\ldots,m$.

\begin{theorem}\label{th:Deiej}
For any $i=1,\ldots,m$, $\cD\veps_i$ is a central, $G/Z$-graded division algebra with support $T/Z$,
and $n_1=\cdots=n_m=\lvert H/Z\rvert=\sqrt{\lvert T/Z\vert}$.
Moreover, for all $1\leq i,j\leq m$, $\cD\veps_i$ and $\cD\veps_j$ are isomorphic as $G/Z$-graded algebras.
\end{theorem}

\begin{proof}
For any $t\in T$, 
\[
\cD_{tZ}\bydef \bigoplus_{z\in Z}\cD_{tz}=\cD_t\cZ=\bigoplus_{i=1}^m\cD_t\veps_i.
\]
Hence $\cD_{tZ}\veps_i=\cD_t\veps_i$ has dimension $1$ for $i=1,\ldots,m$, and if $\Xi$ is a transversal of $Z$ in $T$, then
\begin{equation}\label{eq:Dei}
\cD\veps_i=\bigoplus_{t\in \Xi}\cD_{tZ}\veps_i=\bigoplus_{t\in \Xi}\cD_t\veps_i,
\end{equation}
is $G/Z$-graded with support $T/Z$, with $(\cD\veps_i)_{tZ}=\cD_t\veps_i$ of dimension $1$ for any $t\in T$. 

For any $i=1,\ldots,m$, $t\in T$ and $0\ne x\in\cD_t\veps_i$, there is a nonzero element $y\in \cD_t$ such that $x=y\veps_i$. But the nonzero homogeneous elements in $\cD$ are invertible. Then $x(y^{-1}\veps_i)=(y^{-1}\veps_i)x=\veps_i$, which is the identity element in $\cD\veps_i$. Thus $x$ is invertible in $\cD\veps_i$. Therefore, $\cD\veps_i$ is a $G/Z$-graded division algebra, and it is central because its center is the intersection $\cZ\cap\cD\veps_i=\FF\veps_i$ (alternatively, use that $\cD\veps_i$ is isomorphic to $M_{n_i}(\FF)$). 

Now, for any $i=1,\ldots,m$, we have $n_i^2=\dim \cD\veps_i=\lvert T/Z \rvert$, and Proposition \ref{pr:DTHbeta}(iv) finishes the proof of the first assertion. The second assertion is Proposition~3 in \cite{EK_D4} (see also Remark \ref{re:DLpi1}, below, for an alternative argument).
\end{proof}

\begin{corollary}\label{co:Wei}
For any $i=1,\ldots,m$, $\cW\veps_i$ is a $G/Z$-graded simple submodule of $\cW$. For $i\ne j$, the modules $\cW\veps_i$ and $\cW\veps_j$ are not isomorphic as ungraded modules.
\end{corollary}
\begin{proof}
$\cD\veps_i$ is the centralizer $C(\cW\veps_i)$ of $\cW\veps_i$. Since $\cW\veps_i$ is completely reducible as an ungraded module, it is also completely reducible as a $G/Z$-graded module (see \cite[Lemma 1]{EK_Israel}). If it were not graded simple, then its centralizer would contain proper idempotents in its neutral component $(\cD\veps_i)_e=\FF\veps_i$ (for each decomposition $\cW\veps_i=\cW'\oplus\cW''$, with $\cW'$ and $\cW''$ $G/Z$-graded submodules, the projection on $\cW'$ parallel to $\cW''$ is such an idempotent), which is impossible.
\end{proof}

It turns out that $\cW$ is isomorphic to a suitable loop module of any of the $\cW\veps_i$:

\begin{proposition}\label{pr:WLpi1}
Let $\pi':G\rightarrow G/Z$ be the natural homomorphism. Then, for any $i=1,\ldots,m$, the linear map given by
\[
\begin{split}
\phi_i:\cW&\longrightarrow L_{\pi'}(\cW\veps_i)\\
 w_g&\mapsto w_g\veps_i\otimes g,
\end{split}
\]
is an isomorphism of $G$-graded left $\cR$-modules.
\end{proposition}

\begin{proof}
It is clear that $\phi_i$ is a $G$-graded homomorphism. For any $g\in G$,
\[
\cW_{gZ}=\bigoplus_{z\in Z}\cW_{gz}=\cW_g\cZ=\cW_g\veps_i\oplus\Bigl(\bigoplus_{j\ne i}\cW_g\veps_j\Bigr),
\]
so $\cW_{gZ}\veps_i=\cW_g\veps_i$ and hence $\phi_i\vert_{\cW_g}:\cW_g\rightarrow (\cW\veps_i)_{gZ}\otimes g=\cW_g\veps_i\otimes g$ is surjective. Hence $\phi_i$ is surjective. Since $\cW$ is graded simple, $\phi_i$ is injective, too.
\end{proof}

\begin{remark}\label{re:DLpi1}
With the same argument, one proves that the linear map $\cD\rightarrow L_{\pi'}(\cD\veps_i)$, $d_g\mapsto d_g\veps_i\otimes g$, is an isomorphism of $G$-graded algebras. Hence, for any $1\le i,j\le m$, we have $\cD\veps_i\sim_{\pi'}\cD\veps_j$ (see \cite[Theorem 7.1.1.(ii)]{ABFP}), and, since $\FF$ is algebraically closed, it follows that $\cD\veps_i$ and $\cD\veps_j$ are $G/Z$-graded isomorphic \cite[Lemma 6.3.4]{ABFP}.
\end{remark}

Due to Theorem \ref{th:Deiej}, the next definition makes sense. We will show that it extends the corresponding concepts defined in \cite{EK_Israel} for finite-dimensional modules of  semisimple Lie algebras over an algebraically closed field of characteristic zero.

\begin{df}\label{df:inertia_Brauer}
Let $\cW$ be a $G$-graded simple left $\cR$-module such that $\dim C(\cW)$ is finite and not divisible by the characteristic of $\FF$. 
Denote $\cD=C(\cW)$ and let $\cZ$ be the center of $\cD$ and $Z$ be the support of $\cZ$.
\begin{itemize}
\item The \emph{inertia group} of $\cW$ is $K_{\cW}\bydef\{\chi\in \wh{G}:\chi(z)=1\ \forall z\in Z\}$ (that is, $K_{\cW}=Z^\perp$).

\item The \emph{(graded) Brauer invariant} of $\cW$ is the isomorphism class of the $G/Z$-graded division algebra $\cD\veps$, where $\veps$ is any primitive idempotent of $\cZ$.

\item The \emph{(graded) Schur index} of $\cW$ is the degree of the matrix algebra $\cD\veps$.
\end{itemize}
\end{df}

The Brauer invariant of $\cW$ is an element of $B_{G/Z}(\FF)$, a graded version of Brauer group defined in \cite{PP}. In general, for a field $\FF$ and an abelian group $G$, the group $B_G(\FF)$ consists of the equivalence classes of finite-dimensional associative $\FF$-algebras that are central, simple, and $G$-graded, where $\cA_1\sim\cA_2$ if and only if there exist finite-dimensional $G$-graded $\FF$-vector spaces $\cV_1$ and $\cV_2$ such that $\cA_1\otimes\End(\cV_1)\simeq\cA_2\otimes\End(\cV_2)$ as $G$-graded algebras. Here the symbol $\otimes$ denotes the usual (untwisted) tensor product of $\FF$-algebras, equipped with the natural $G$-grading. This tensor product induces a group structure on the set of equivalence classes: $[\cA_1][\cA_2]\bydef[\cA_1\otimes\cA_2]$. 

Every class $[\cA]$ contains a unique graded division algebra (up to isomorphism). Indeed, there exist a $G$-graded division algebra $\cD$ and a $G$-graded right $\cD$-module $\cW$ such that $\cA$ is graded isomorphic to $\End_\cD(\cW)$ (see e.g. \cite[Theorem 2.6]{EK_mon}), and such $\cD$ is unique up to graded isomorphism (\cite[Theorem 2.10]{EK_mon}). Since $\cD$ is a graded division algebra, we can find a $\cD$-basis $\{w_1,\ldots,w_s\}$ of $\cW$ that consists of homogeneous elements. Let $\wt{\cW}=\FF w_1\oplus\cdots\oplus\FF w_s$. Then $\wt{\cW}$ is a $G$-graded vector space, and the map
\[
\wt{\cW}\otimes_\FF\cD\to\cW,\;
w\otimes d\mapsto wd,
\]
is a graded isomorphism. Thus we can assume $\cW=\wt{\cW}\otimes_\FF\cD$ and hence identify
\begin{equation}\label{eq:EndDW}
\End_\cD(\cW)\simeq\End_\FF(\wt{\cW})\otimes_\FF\cD.
\end{equation}
Now the isomorphism $\cA\simeq\End(\wt{\cW})\otimes\cD$ implies $[\cA]=[\cD]$, and the uniqueness of $\cD$ mentioned above implies that $[\cD_1]=[\cD_2]$ if and only if $\cD_1\simeq\cD_2$ as graded algebras.

In general, the classical Brauer group $B(\FF)$ is contained in $B_G(\FF)$ as the classes of division algebras with trivial $G$-grading. If $\FF$ is algebraically closed (as we assume in this section) then, for any abelian group $G$, the Brauer group $B_G(\FF)$ is isomorphic to the group of alternating continuous bicharacters of the pro-finite group $\wh{G_0}$ where $G_0$ is the torsion subgroup of $G$ if $\chr{\FF}=0$ and the $p'$-torsion subgroup of $G$ if $\chr{\FF}=p>0$ (that is, the set of all elements whose order is finite and coprime with $p$) --- see \cite[\S 2]{EK_Israel}. 

The next result shows that the inertia group is determined by any simple (ungraded) submodule of $\cW$, and our present definition agrees with the one given in \cite{EK_Israel}.

\begin{proposition}\label{pr:inertia}
Let $(\cW,\cF)$ be an object of $\frN(\pi)$ and let $\cV$ be a simple (ungraded) submodule of $\cW$. Then the inertia group of $\cW$ is given by
\[
K_\cW=\{\chi\in\wh{G}: \cV^{\alpha_{\chi}}\ \text{is isomorphic to}\ \cV\}.
\]
\end{proposition}
\begin{proof}
We may assume, without loss of generality, that $\cV=\cV^1$ in Theorem \ref{th:Wcomp.red.}. Then the result follows from Corollary \ref{co:V1alpha}.
\end{proof}

In the next section, we will show that the Brauer invariant, too, is determined by any simple submodule of $\cW$, under the assumption that $\cW$ is finite-dimensional.


\section{Graded simple modules with simple centralizers}\label{se:simple_centralizer}

\emph{As in the previous section, we assume here that $H$ is finite and $\FF$ is algebraically closed of characteristic not dividing $\lvert H\rvert$.}

Let $(\cW,\cF)$ be an object of $\frN(\pi)$. Then $\cD=C(\cW)$ is finite-dimensional (Corollary \ref{co:WDfinite}). Let $\cZ$ be the center of $\cD$. As we have seen in the previous section --- in particular, Theorem \ref{th:Deiej}, Proposition \ref{pr:WLpi1} and Remark \ref{re:DLpi1} --- the study of  $(\cW,\cF)$ reduces to $(\cW\veps,\cF\veps)$, which is an object of $\frN(\pi'')$ by Corollary \ref{co:Wei}, where $\veps$ is any primitive idempotent of $\cZ$, $Z$ is the support of $\cZ$, and $\pi''$ is the natural homomorphism $G/Z\to G/H$ (with kernel $H/Z$). 

So, for the most part of this section, we will assume that $\cZ=\FF 1$, or, equivalently, that $\cD$ is simple --- see Proposition \ref{pr:DTHbeta} --- and will return to the general case at the end. Under this assumption, $\cD$ is isomorphic to the twisted group algebra $\FF^\sigma T$ where $T$ is the support of $\cD$ and $\sigma:T\times T\rightarrow \FF^\times$ is a $2$-cocycle such that $\sigma(h_1,h_2)=1$, for all $h_1,h_2\in H$, and the alternating bicharacter $\beta:T\times T\rightarrow \FF^\times$, defined by $\beta(t_1,t_2)=\sigma(t_1,t_2)\sigma(t_2,t_1)^{-1}$, is nondegenerate. It follows that $\lvert T\rvert=\lvert H\rvert^2$ and the characteristic of $\FF$ does not divide $\lvert T\rvert$.

First we will construct the simple (ungraded) left and right $\cD$-modules, as well as the associated Morita context relating $\FF$ and $\cD$, in terms of the model $\cD=\FF^\sigma T$. It will be apparent from our construction that these simple modules admit a compatible $T/H$-grading. As before, we denote by $c_t$ the basis elements of $\FF^\sigma T$ so that $c_{t_1}c_{t_2}=\sigma(t_1,t_2)c_{t_1t_2}$ for all $t_1,t_2\in T$.

\begin{proposition}\label{pr:TM}
Let $\wb{T}=T/H$ and write $\bar t=tH$ for $t\in T$. Pick a section $\xi:\wb{T}\rightarrow T$ of the natural homomorphism $T\to\wb{T}$.
\begin{romanenumerate}
\item Let $\cM$ be a vector space with basis $e_{\bar t}$ labeled by the elements $\bar t\in\wb{T}$ (so $\dim\cM=\lvert T/H\rvert=\lvert H\rvert$). Define a left $\cD$-action on $\cM$ by
\begin{equation}\label{eq:t1xt2}
c_{t_1}\cdot e_{\bar t_2}
=\sigma\bigl(t_1,\xi(\bar t_2)\bigr)
 \sigma\bigl(\xi(\overline{t_1t_2}),t_1\xi(\bar t_2)\xi(\overline{t_1t_2})^{-1}\bigr)^{-1}e_{\overline{t_1t_2}},
\end{equation}
for all $t_1,t_2\in T$. 
Then $\cM$ is the unique, up to isomorphism, simple left $\cD$-module.
Moreover, $\cM$ becomes a $\wb{T}$-graded left $\cD$-module if we declare $e_{\bar t}$ to have degree $\bar t$.

\item Let $\varrho:\cD\rightarrow \End_\FF(\cM)$ be the homomorphism determined by the left $\cD$-action. Then $\varrho$ is an isomorphism of $\wb{T}$-graded algebras
(with respect to the $\wb{T}$-grading on $\cD$ induced by the natural homomorphism $T\to\wb{T}$).

\item Let $\cM^*=\Hom_\FF(\cM,\FF)$ with a right $\cD$-action given by
\[
(f\cdot d)(x)\bydef f(d\cdot x),
\]
for all $f\in\cM^*$, $d\in\cD$ and $x\in\cM$. Then $\cM^*$ is the unique, up to isomorphism, simple right $\cD$-module. Moreover, the $\wb{T}$-grading on $\cM$ determines a $\wb{T}$-grading on $\cM^*$ such that the evaluation $\cM^*\otimes_\cD\cM\to\FF$ is a graded map (with respect to the trivial $\wb{T}$-grading on $\FF$).

\item The linear map
\[
\begin{split}
\cM\otimes_\FF\cM^*&\longrightarrow \cD\\
  x\otimes f&\mapsto \varrho^{-1}(x.f),
\end{split}
\]
is an isomorphism of $\wb{T}$-graded $(\cD,\cD)$-bimodules, where $x.f:\cM\rightarrow\cM$ is the map $y\mapsto f(y)x$, for all $x,y\in\cM$ and $f\in\cM^*$.
\end{romanenumerate}
\end{proposition}

\begin{proof} 
Let $\cF$ be the subalgebra of $\cD$ spanned by the elements $c_h$, $h\in H$. Since the $2$-cocycle $\sigma$ is trivial on $H$, we have $\cF\simeq\FF H$. Actually, $\cF$ is a maximal graded subfield of $\cD$. Let $\rho:\cF\rightarrow \FF$ the augmentation map: $\rho(c_h)=1$ for all $h\in H$. Clearly, for any $t\in T$, we have
\[
\cD_{tH}=\bigoplus_{h\in H}\cD_{th}=\cD_{\xi(\bar t)}\cF,
\]
and
\[
\cD=\bigoplus_{\bar t\in\wb{T}}\cD_{\xi(\bar t)H}
=\bigoplus_{\bar t\in\bar T}\cD_{\xi(\bar t)}\cF.
\]
As a left $\cD$-module, $\cD$ is $T$-graded simple with centralizer equal to $\cD$ (right action), and $\cF$ is a maximal graded subfield of the centralizer. By Theorem \ref{th:specializations}, $\cD\ker(\rho)$ is a maximal submodule of $\cD$. Since $\cF=\FF 1\oplus\ker(\rho)$, we obtain
\[
\cD=\bigoplus_{\bar t\in\wb{T}}\cD_{\xi(\bar t)}\cF
 =\Bigl(\bigoplus_{\bar t\in\wb{T}}\cD_{\xi(\bar t)}\Bigr)\oplus \cD\ker(\rho),
\]
and, similarly to the proof of Proposition \ref{pr:twisted}, we may identify the simple $\cD$-module $\cD/\cD\ker(\rho)$ with 
$\bigoplus_{\bar t\in\wb{T}}\cD_{\xi(\bar t)}$ where the $\cD$-action is given by Equation \eqref{eq:rgwk} as follows:
\[
c_{t_1}\cdot c_{\xi(\bar t_2)}\bydef\text{projection of $c_{t_1}c_{\xi(\bar t_2)}$ on $\bigoplus_{\bar t\in\wb{T}}\cD_{\xi(\bar t)}$ parallel to $\cD\ker(\rho)$},
\]
for any $t_1,t_2\in T$.
Now, $c_{t_1}c_{\xi(\bar t_2)}=\sigma\bigl(t_1,\xi(\bar t_2)\bigr)c_{t_1\xi(\bar t_2)}$ and there is a unique $h\in H$ such that $t_1\xi(\bar t_2)=\xi(\overline{t_1t_2})h$, so
\[
\begin{split}
c_{t_1}c_{\xi(\bar t_2)}&=\sigma\bigl(t_1,\xi(\bar t_2)\bigr)c_{\xi(\overline{t_1t_2})h}\\
&=\sigma\bigl(t_1,\xi(\bar t_2)\bigr)\sigma\bigl(\xi(\overline{t_1t_2}),h\bigr)^{-1}c_{\xi(\overline{t_1t_2})}c_h\\
&=\sigma\bigl(t_1,\xi(\bar t_2)\bigr)\sigma\bigl(\xi(\overline{t_1t_2}),h\bigr)^{-1}c_{\xi(\overline{t_1t_2})}\\
&\qquad+\sigma\bigl(t_1,\xi(\bar t_2)\bigr)\sigma\bigl(\xi(\overline{t_1t_2}),h\bigr)^{-1}c_{\xi(\overline{t_1t_2})}(c_h-1)\\
&\in \Bigl(\bigoplus_{\bar t\in\wb{T}}\cD_{\xi(\bar t)}\Bigr)\oplus \cD\ker(\rho).
\end{split}
\]
Therefore,
\[
\begin{split}
c_{t_1}\cdot c_{\xi(\bar t_2)}&=\sigma\bigl(t_1,\xi(\bar t_2)\bigr)\sigma\bigl(\xi(\overline{t_1t_2}),h\bigr)^{-1}c_{\xi(\overline{t_1t_2})}\\
&=\sigma\bigl(t_1,\xi(\bar t_2)\bigr)\sigma\bigl(\xi(\overline{t_1t_2}),t_1\xi(\bar t_2)\xi(\overline{t_1t_2})^{-1}\bigr)^{-1}c_{\xi(\overline{t_1t_2})}.
\end{split}
\]
Thus, the linear map given by $c_{\xi(\bar t)}\mapsto e_{\bar t}$, for all $\bar t\in\wb{T}$, is an isomorphism of $\cD$-modules $\cD/\cD\ker(\rho)\rightarrow \cM$.
All remaining assertions are clear because, as an ungraded algebra, $\cD$ is isomorphic to the matrix algebra $M_n(\FF)$ where $n=\lvert H\rvert$. 
\end{proof}

Given a graded division algebra $\cD$ that is finite-dimensional and simple, there are ways to choose the $2$-cocycle $\sigma$ (within the same cohomology class), the subgroup $H$ of the support $T$, and the section $\xi:T/H\to T$ that make Equation \eqref{eq:t1xt2} in Proposition \ref{pr:TM} much simpler. The alternating bicharacter $\beta$ is nondegenerate, so there are maximal isotropic subgroups $A$ and $B$ of $T$ such that $T=A\times B$ (see e.g. \cite[Section 2.2]{EK_mon}). Then we may multiply the basis elements $c_a$, $a\in A$, and $c_b$, $b\in B$, by suitable nonzero scalars in such a way that
\[
c_{a_1}c_{a_2}=c_{a_1a_2}\quad\text{and}\quad c_{b_1}c_{b_2}=c_{b_1b_2}
\]
for all $a_1,a_2\in A$ and $b_1,b_2\in B$. For any $a\in A$ and $b\in B$ define $c_{ab}\bydef c_ac_b$. Then, for all $a_1,a_2\in A$ and $b_1,b_2\in B$, we have 
\[
c_{a_1b_1}c_{a_2b_2}=c_{a_1}c_{b_1}c_{a_2}c_{b_2}=\beta(b_1,a_2)c_{a_1a_2}c_{b_1b_2}=\beta(b_1,a_2)c_{a_1a_2b_1b_2}.
\]
Therefore, $\cD$ is isomorphic to the twisted group algebra $\FF^\sigma T$ where the $2$-cocycle is given by $\sigma(a_1b_1,a_2b_2)=\beta(b_1,a_2)$.

Now we take $H=A$ and the `natural' section $\xi(\overline{ab})=b$, for all $a\in A$ and $b\in B$. Then the simple module $\cM$ becomes
\[
\cM=\bigoplus_{b\in B}\FF e_b,
\]
and Equation \ref{eq:t1xt2} simplifies as follows:
\[
\begin{split}
c_{ab}\cdot e_{b'}
 &=\sigma(ab,b')\sigma\bigl(bb',abb'(bb')^{-1}\bigr)^{-1}e_{bb'}\\
 &=\sigma(ab,b')\sigma(bb',a)^{-1}e_{bb'}\\
 &=\beta(bb',a)^{-1}e_{bb'}\\
 &=\beta(a,bb')e_{bb'},
\end{split}
\]
for all $a\in A$ and $b,b'\in B$. 

To summarize: we may identify $\cD$ with the `smash product' $\FF A\#\FF B$, that is, vector space $\FF A\otimes\FF B$ where multiplication is given by
\[
(a_1\otimes b_1)(a_2\otimes b_2)=\beta(b_1,a_2)(a_1a_2\otimes b_1b_2),
\] 
and $\cM$ with vector space $\FF B$ where the left $\cD$-action is given by
\[
(a\otimes b)\cdot b'=\beta(a,bb')\,bb',
\]
for all $a,a_1,a_2\in A$ and $b,b',b_1,b_2\in B$. (This model of $\cM$ appeared in \cite[Remark 18]{EK_Israel}.)
We also observe that the dual right $\cD$-module $\cM^*$ can be identified with $\FF A$ through the linear map $\FF A\rightarrow \cM^*$, $a\mapsto f_a$, where $f_a(b)=\beta(a,b)$. The right $\cD$-action on $\FF A$ is given by
\[
a'\cdot(a\otimes b)=\beta(aa',b)\,aa',
\]
for all $a,a'\in A$ and $b\in B$.

Note that not every maximal isotropic subgroup of $(T,\beta)$ admits a maximal isotropic complement. 

\begin{example}
Assuming $\chr\FF\ne 2$, let $T=\ZZ_4\times\ZZ_4$ and $\beta\bigl((x_1,y_1),(x_2,y_2)\bigr)=i^{y_1x_2}$ for $x_1,x_2,y_1,y_2\in\ZZ_4=\ZZ/4\ZZ$, where $i=\sqrt{-1}$. Then the maximal isotropic subgroup $2\ZZ/4\ZZ\times 2\ZZ/4\ZZ$ does not admit any complement.
\end{example}

Our main result in this section is the following:

\begin{theorem}\label{th:Dcentral}
Let $(\cW,\cF)$ be a finite-dimensional object in $\frN(\pi)$ such that $C(\cW)$ is simple. Let $\rho:\cF\rightarrow \FF$ be a homomorphism of unital algebras and let $\cV=\cW/\cW\ker(\rho)$ be the canonical central image. Denote by $\varrho_\cV:\cR\rightarrow\End_\FF(\cV)$ the associated representation.
Then there exists a unique $G$-grading on $\End_\FF(\cV)$ such that $\varrho_\cV$ becomes a $G$-graded homomorphism, and the graded Brauer invariant of $\cW$ (see Definition \ref{df:inertia_Brauer}) is precisely $[\End_\FF(\cV)]$.
\end{theorem}

\begin{proof}
Since  $\cD=C(\cW)$ is simple, the graded Brauer invariant of $\cW$ is just $[\cD]$. 
Since $\cW$ is finite-dimensional and graded simple, by graded density (Theorem \ref{th:graded_central}), the image of the associated representation $\varrho_\cW:\cR\rightarrow \End_\FF(\cW)$ is equal to $\End_\cD(\cW)$. Thus, we may regard $\varrho_\cW$ as a surjective $G$-graded homomorphism $\cR\rightarrow \End_\cD(\cW)$. Pick a homogeneous $\cD$-basis of $\cW$ and let $\wt{\cW}$ be its $\FF$-span. Using the identification $\cW=\wt{\cW}\otimes_\FF\cD$, we obtain  $\End_\cD(\cW)\simeq\End_\FF(\tilde{\cW})\otimes_\FF\cD$, as in Equation \eqref{eq:EndDW}. Also,
\[
\cW\ker(\rho)=(\wt{\cW}\otimes_\FF\cD)\ker(\rho)=\wt{\cW}\otimes_\FF\bigl(\cD\ker(\rho)\bigr).
\]
On the other hand, $\cD\simeq\FF^\sigma T$, so we may apply Proposition \ref{pr:TM} and consider the simple $G/H$-graded left $\cD$-module $\cM\simeq \cD/\cD\ker(\rho)$ defined there. The composition
\[
\wt{\cW}\otimes_\FF\cD
\twoheadrightarrow\wt{\cW}\otimes_\FF\bigl(\cD/\cD\ker(\rho)\bigr)
\xrightarrow{\sim}\wt{\cW}\otimes_\FF\cM
\]
is a $G/H$-graded homomorphism from $\cW$ onto $\wt{\cW}\otimes_\FF\cM$ with kernel $\wt{\cW}\otimes_\FF\cD\ker(\rho)=\cW\ker(\rho)$. Hence, $\cV=\cW/\cW\ker(\rho)$ is $G/H$-graded isomorphic to $\wt{\cW}\otimes_\FF\cM$.

Since the $\cR$-module $\cV$ is finite-dimensional and simple, the associated representation $\varrho_\cV$ is surjective. According to the above analysis, it can be obtained as the composition of $\cR\xrightarrow{\varrho_\cW}\End_\cD(\cW)$ and the isomorphism $\End_\cD(\cW)\xrightarrow{\Phi}\End_\FF(\cV)$, where $\Phi$ is the composition
\[
\End_\cD(\cW)\xrightarrow{\sim}\End_\FF(\wt{\cW})\otimes\cD
\xrightarrow{\sim}\End_\FF(\wt{\cW})\otimes\End_\FF(\cM)
\xrightarrow{\sim}\End_\FF(\wt{\cW}\otimes\cM)
\xrightarrow{\sim}\End_\FF(\cV).
\]
The isomorphism $\Phi$ induces on $\End_\FF(\cV)$ a $G$-grading, which is a refinement of the $G/H$-grading coming from $\cV$. Since $\varrho_\cW$ is $G$-graded, $\varrho_\cV=\Phi\circ\varrho_\cW$ becomes $G$-graded, too. As $\varrho_\cV$ is surjective, this grading on $\End_\FF(\cV)$ is uniquely determined: $\End_\FF(\cV)_g=\varrho_\cV(\cR_g)$ for all $g\in G$. The $G$-graded isomorphism $\Phi$ shows that $[\cD]=[\End_\FF(\cV)]$.
\end{proof}

\begin{corollary}\label{co:Dcentral} 
Let $\cW$ be a finite-dimensional $G$-graded simple left $\cR$-module such that the characteristic of $\FF$ does not divide the dimension of $C(\cW)$. Let $\cV$ be a simple (ungraded) submodule of $\cW$ and let $\varrho_\cV:\cR\rightarrow\End_\FF(\cV)$ be the associated representation. Let $Z$ be the support of the center of $C(\cW)$. 
\begin{romanenumerate}
\item There is a unique $G/Z$-grading on $\End_\FF(\cV)$ that makes $\varrho_\cV$ a $G/Z$-graded homomorphism.

\item The graded Brauer invariant of $\cW$ is precisely $[\End_\FF(\cV)]$. 

\item For any subgroup $H$ of the support of $C(\cW)$, maximal relative to the property of $\cF\bydef\bigoplus_{h\in H}C(\cW)_h$ being commutative, $\cV$ is endowed with a structure of $G/H$-graded $\cR$-module such that $\cV$ is a central image of $(\cW,\cF)$.
\end{romanenumerate}
\end{corollary}

\begin{proof}
Let $\veps_1,\ldots,\veps_m$ be the primitive idempotents of the center $\cZ$ of $C(\cW)$. Then $\cV$ is a simple submodule of $\cW\veps_i$ for a unique $i=1,\ldots,m$. But $\cW\veps_i$ is a $G/Z$-graded simple left $\cR$-module (Corollary \ref{co:Wei}) with centralizer $C(\cW)\veps_i$, which is simple. The graded Brauer invariant of $\cW$ is the $G/Z$-graded isomorphism class of $C(\cW)\veps_i$, and Theorem \ref{th:Dcentral} gives (i) and (ii).

Part (iii) follows from the fact that $\cV$ is isomorphic to $\wt{\cW}\otimes_\FF\cM$ (see the proof of Theorem \ref{th:Dcentral}) where $\cM$ is a simple left $C(\cW)\veps_i$-module graded by $(G/Z)/(H/Z)\simeq G/H$.
\end{proof}

In particular, Corollary \ref{co:Dcentral} shows that our definition of graded Brauer invariant agrees with the one given in \cite{EK_Israel} for the case considered therein.

\smallskip

Under the hypotheses of Theorem \ref{th:Dcentral}, $(\cW,\cF)$ is isomorphic to $\bigl(L_\pi(\cV),L_\pi(\FF 1)\bigr)$ in $\frN(\pi)$ (Theorem \ref{th:MpiNpi}). Also, as an ungraded module, $\cW\simeq\wt{\cW}\otimes\cD\simeq (\wt{\cW}\otimes\cM)\otimes\cM^*$ is a direct sum of $\dim\cM^*=\lvert H\rvert$ copies of $\cV\simeq \wt{\cW}\otimes\cM$.

Let us now drop the assumption that $\cD=C(\cW)$ is simple. Then, for any primitive idempotent $\veps$ in the center $\cZ$ of $\cD$, $\cW\veps$ is $G/Z$-graded simple, where $Z$ is the support of $\cZ$. Thus, $(\cW\veps,\cF\veps)$ is isomorphic to $\bigl(L_{\pi''}(\cV),L_{\pi''}(\FF 1)\bigr)$, where $\pi'':G/Z\rightarrow G/H$ is the natural homomorphism. Proposition~\ref{pr:WLpi1} and Remark~\ref{re:DLpi1} show that $(\cW,\cF)$ is isomorphic to $\bigl(L_{\pi'}(\cW\veps),L_{\pi'}(\cF\veps)\bigr)$, where $\pi':G\to G/Z$ is the natural homomorphism. Now Proposition \ref{pr:transitive} shows that $(\cW,\cF)$ is isomorphic to $\bigl(L_\pi(\cV),L_\pi(\FF 1)\bigr)$ where $\pi=\pi''\circ\pi':G\to G/H$. Of course, Theorem \ref{th:MpiNpi} gives this isomorphism immediately, but it is natural to proceed in two steps, through $\pi''$ and $\pi'$. These two steps are clearly separated in \cite{EK_Israel}.


\section{Finite-dimensional graded simple modules in characteristic zero}\label{se:char0}

\emph{In this last section, $\FF$ will be assumed algebraically closed of characteristic zero.}

If $\cW$ is a finite-dimensional $G$-graded simple left $\cR$-module and $Z$ is the support of the center $\cZ$ of $C(\cW)$ then $\cW$ decomposes as in Equation \eqref{eq:WWeis}: 
\[
\cW=\cW\veps_1\oplus\cdots\oplus\cW\veps_m,
\] 
where $\veps_1,\ldots,\veps_m$ are the primitive idempotents of $\cZ$. These $\cW\veps_i$ are $G/Z$-graded simple submodules that are not isomorphic to each other (Corollary \ref{co:Wei}).

In turn, each $\cW\veps_i$ is a direct sum of copies of a simple module $\cV^i$. The simple submodules of $\cW$ are, up to isomorphism, obtained by twisting one of these by the automorphisms $\alpha_\chi$ given by Equation \eqref{eq:alpha_chi}, for $\chi\in\wh{G}$ (Corollary \ref{co:V1alpha}).

There arises the natural question of determining the finite-dimensional simple $\cR$-modules that appear as simple submodules of $G$-graded simple $\cR$-modules.

\begin{theorem}\label{th:fin.ind.} 
Let $\cV$ be a finite-dimensional simple left $\cR$-module. Consider the subgroup
\[
K_\cV\bydef\{\chi\in \wh{G}: \cV^{\alpha_\chi}\ \text{is isomorphic to}\ \cV\}.
\]
Then $\cV$ is isomorphic to a simple submodule of a finite-dimensional $G$-graded simple left $\cR$-module if and only if the index $[\wh{G}:K_\cV]$ is finite.
\end{theorem}

\begin{proof}
If $\cV$ is isomorphic to a simple submodule of a finite-dimensional $G$-graded simple module $\cW$, then Proposition \ref{pr:inertia} shows that $K_\cV$ is the inertia group of $\cW$, and hence coincides with $Z^\perp$, where $Z$ is the support of the center of $C(\cW)$. Then $[\wh{G}:K_\cV]=\lvert Z\rvert$ is finite. 

Conversely, assume that $\cV$ is a finite-dimensional simple module and $[\wh{G}:K_\cV]$ is finite. Let $\varrho_\cV:\cR\rightarrow \End_\FF(\cV)$ be the associated representation. The simplicity of $\cV$ implies that $\varrho_\cV$ is surjective. Consider the subgroup $Z$ of $G$ orthogonal to $K_\cV$:
\[
Z\bydef\{g\in G: \chi(g)=1\ \forall \chi\in K_\cV\}.
\]
Then $\lvert Z\rvert=[\wh{G}:K_\cV]$ is finite and $Z^\perp=K_\cV$. (If $G$ is finitely generated --- as we may assume ---  then any subgroup of finite index in $\wh{G}$ is Zariski closed because it contains the connected component of $\wh{G}$:  indeed, the connected component is a torus, so it is a divisible abelian group and hence has no nontrivial finite quotients.) For any $\chi\in K_\cV$, $\cV^{\alpha_\chi}$ is isomorphic to $\cV$, so there is a linear isomorphism $\varphi_\chi:\cV\rightarrow \cV$ such that
\[
\varphi_\chi(rv)=\alpha_\chi(r)\varphi_\chi(v),
\]
for all $r\in\cR$ and $v\in \cV$. By Schur's Lemma, $\varphi_\chi$ is unique up to scalar multiplication. Hence, for $\chi_1,\chi_2\in K_\cV$, $\varphi_{\chi_1}\varphi_{\chi_2}$ and $\varphi_{\chi_1\chi_2}$ are equal up to a scalar factor. Therefore, the map
\[
\begin{split}
K_\cV&\longrightarrow \Aut\bigl(\End_\FF(\cV)\bigr)\\
 \chi&\mapsto \Ad \varphi_\chi: f\mapsto \varphi_\chi f\varphi_\chi^{-1},
\end{split}
\]
is a group homomorphism.

Note that, for any $g\in G$ and $r_g\in\cR_g$,
\[
\varphi_\chi\varrho_\cV(r_g)\varphi_\chi^{-1}=\varrho_\cV\bigl(\alpha_\chi(r_g)\bigr)
=\chi(g)\varrho_\cV(r_g),
\]
so the linear endomorphisms $\Ad\varphi_\chi$ are simultaneously diagonalizable on $\End_\FF(\cV)=\varrho(\cR)$. By our assumptions on $\FF$, $K_\cV$ separates points of $G/Z$, so there is a unique $G/Z$-grading on $\End_\FF(\cV)$ such that, for any $g\in G$,
\[
\End_\FF(\cV)_{gZ}=\{f\in\End_\FF(\cV):\varphi_\chi f\varphi_\chi^{-1}=\chi(g)g\ \forall\chi\in K_\cV\},
\]
and $\varrho_\cV$ becomes a $G/Z$-graded homomorphism.

By \cite[Theorem 2.6]{EK_mon}, there is a $G/Z$-graded division algebra $\cD'$ and a $G/Z$-graded right $\cD'$-module $\cW'$ such that $\End_\FF(\cV)$ is $G/Z$-graded isomorphic to $\End_{\cD'}(\cW')$, hence the composition
\[
\cR\xrightarrow{\varrho_\cV}\End_\FF(\cV)\xrightarrow{\sim}\End_{\cD'}(\cW')
\]
endows $\cW'$ with a structure of $G/Z$-graded simple left $\cR$-module. (The graded simplicity follows by the surjectivity of $\varrho_\cV$.)
Besides, any simple submodule of $\cW'$ is isomorphic to $\cV$, because $\cV$ is, up to isomorphism, the only simple module for $\End_\FF(\cV)$.

Now the arguments in \cite[\S 3.3]{EK_Israel} show that the induced module $\cW\bydef I_{\pi'}(\cW')$ (see Proposition \ref{pr:loop_induced}), for the natural homomorphism $\pi':G\to G/Z$, is a $G$-graded simple left $\cR$-module, and $\cV$ is isomorphic to a simple submodule of $\cW$.
\end{proof}

\begin{remark}
By construction, $\cD'=C(\cW')$ is the graded Brauer invariant of $\cW$ and $C(\cW)\simeq L_{\pi'}(\cD')$ (Corollary \ref{co:Dcentral} and Remark \ref{re:DLpi1}). Let $T'$ be the support of $\cD'$, so the support of $C(\cW)$ is the inverse image $T$ of $T'$ under $\pi'$, and let $\beta'$ be the corresponding nondegenerate alternating bicharacter $T'\times T'\to\FF^\times$. Fix a maximal isotropic subgroup $H'\subset T'$ for $\beta'$, and let $\cF'=\bigoplus_{h\in H'}\cD'_h$ and $\cF=L_{\pi'}(\cF')$. Then $(\cW,\cF)$ is an object of $\frN(\pi)$, where $\pi$ is the natural homomorphism $G\to G/H$, and $(\cW,\cF)\simeq (L_{\pi}(\cV),L_\pi(\FF 1))$, where $\cV$ is equipped with a grading by $G/H$ to make it a graded $\cR$-module. Since such a grading is unique up to isomorphism and shift, the loop functor $L_\pi$ gives a bijection between, on the one hand, the classes of finite-dimensional $G$-graded simple modules under isomorphism and shift and, on the other hand, the finite $\wh{G}$-orbits of isomorphism classes of finite-dimensional simple modules (cf. Remark \ref{rm:classification}). Finally, note that $\cW$ and $\cW^{[g]}$ are isomorphic if and only if $g\in T$.
\end{remark}

In \cite{EK_Israel}, the finite-dimensional modules for a finite-dimensional $G$-graded semisimple Lie algebra $\cL$ over and algebraically closed field of characteristic zero are studied. Any such module is a left module for $\cR=U(\cL)$, the universal enveloping algebra, which inherits a $G$-grading from $\cL$. In this case, any finite-dimensional simple $\cR$-module $\cV$ satisfies that the index $[\wh{G}:K_\cV]$ is finite, due to the fact that the $G$-grading on $\cL$ is determined by a group homomorphism 
\[
\alpha:\wh{G}\to \Aut(\cL),\;
 \chi \mapsto\ \alpha_\chi,
\]
which we can compose with the natural homomorphism $\Aut(\cL)\to\Aut(\cL)/\inaut(\cL)$. This latter group is finite, being isomorphic to the group of automorphisms of the Dynkin diagram of $\cL$, and hence $[\wh{G}:\alpha^{-1}(\inaut(\cL))]$ is finite. But $\alpha^{-1}(\inaut(\cL))$ is contained in $K_\cV$ for any $\cV$. 

This is the reason behind the fact that any finite-dimensional simple $\cL$-module $\cV$ is isomorphic to a submodule of a $G$-graded simple module. In \cite{EK_Israel}, $\cV$ is not endowed with the structure of graded module by a quotient of $G$. Corollary \ref{co:Dcentral} shows how to do so.

%

\bibliographystyle{amsplain}

\begin{thebibliography}{ABFP08}

\bibitem[ABFP08]{ABFP} B.~Allison, S.~Berman, J.~Faulkner, and A.~Pianzola, \emph{Realization of 
graded-simple algebras as loop algebras}, Forum Math. \textbf{20} (2008), no.~3, 395--432.

\bibitem[BL07]{BL} Y.~Billig and M.~Lau, \emph{Thin coverings of modules}, J.~Algebra \textbf{316} (2007), no.~1, 147--173.

\bibitem[EK13]{EK_mon} 
A.~Elduque and M.~Kochetov, \emph{Gradings on simple {L}ie algebras}, Mathematical Surveys and Monographs \textbf{189},  American Mathematical Society, Providence, RI, 2013.

\bibitem[EK15a]{EK_Israel}
A.~Elduque and M.~Kochetov, \emph{Graded modules over classical simple Lie algebras with a grading}, Israel J. Math. \textbf{207} (2015), no.~1, 229--280.

\bibitem[EK15b]{EK_D4}
A.~Elduque and M.~Kochetov, \emph{Gradings on the Lie algebra $D_4$ revisited}, J.~Algebra \textbf{441} (2015), 441--474.

\bibitem[MZpr]{MZ}
V.~Mazorchuk and K.~Zhao, \emph{Graded simple Lie algebras and graded simple representations}, arXiv:1504.05114v3 [math.RT].

\bibitem[Pas85]{Passman} 
D.S.~Passman, \emph{The algebraic structure of group rings}, Robert E.~Krieger Publishing Co., Inc., Melbourne, FL, 1985.

\bibitem[PP70]{PP}
D.J.~Picco and M.I.~Platzeck, \emph{Graded algebras and Galois extensions}, 
Collection of articles dedicated to Alberto Gonz\'alez Dom\'{\i}nguez on his sixty-fifth birthday, 
Rev. Un. Mat. Argentina \textbf{25} (1970/71), 401--415. 

\bibitem[Smi97]{S}
O.N.~Smirnov, \emph{Simple associative algebras with finite $\ZZ$-grading}, J.~Algebra \textbf{196} (1997), no.~1, 171--184.

\end{thebibliography}

\end{document}